\documentclass[10pt,reqno]{amsart}
\RequirePackage[OT1]{fontenc}
\RequirePackage{amsthm,amsmath}
\RequirePackage[numbers]{natbib}
\RequirePackage[colorlinks,linkcolor=blue,citecolor=blue,urlcolor=blue]{hyperref}
\usepackage{amssymb}
\usepackage{anysize}
\usepackage{hyperref}
\usepackage{extarrows}
\usepackage{multirow}
\usepackage{natbib}
\usepackage{stmaryrd}
\usepackage{amsmath}
\usepackage{enumitem}
\usepackage{mathtools} 
\usepackage{dsfont} 
\usepackage{verbatim}
\usepackage{amsfonts}
\usepackage{amsmath}
\usepackage{amssymb}
\usepackage{amscd}
\usepackage{amsthm}
\usepackage{latexsym,bm}
\usepackage{pstricks}
\usepackage{geometry}
\usepackage{dsfont}
\usepackage{bigints}
\usepackage{bm}
\usepackage{soul}

\numberwithin{equation}{section}
\theoremstyle{plain} 
\newtheorem{theorem}{Theorem}[section]
\newtheorem{lemma}[theorem]{Lemma}

\newtheorem{proposition}[theorem]{Proposition}
\newtheorem{assumption}[theorem]{Assumption}
\newtheorem{definition}[theorem]{Definition}

\theoremstyle{remark}
\newtheorem{remark}[theorem]{Remark}

\renewcommand{\Re}{\mathrm{Re}\,}
\renewcommand{\Im}{\mathrm{Im}\,}

\newcommand{\im}{\mathrm{Im}\,}

\def\be{\begin{equation}}
\def\ee{\end{equation}}
\def\ba{\begin{eqnarray*}}
\def\ea{\end{eqnarray*}}
\def\bae{\begin{eqnarray}}
\def\eae{\end{eqnarray}}
\def\bc{\begin{center}}
\def\ec{\end{center}}
\def\i{\mathrm{i}}

\def\Tr{\mathrm{Tr}}
\def\C{\mathbb{C}}
\def\N{\mathbb{N}}

\def\R{\mathbb{R}}

\def\E{\mathbb{E}}

\def\vi{{\bm v_i}}
\def\si{{\bm s_i}}
\def\ti{{\bm t_i}}
\def\ri{{\bm r_i}}
\def\hi{{\bm h_i}}
\def\ei{{\bm e_i}}
\def\wi{{\bm w_i}}

\def\ek{{\bm e_k}}
\def\ej{{\bm e_j}}

\def\gi{{\bm g_i}}
\def\xii{{\bm x_i}}
\def\yi{{\bm y_i}}
\def\yone{{\bm y}_1}
\def\ytwo{{\bm y}_2}

\def\g{\underline{G}}

\def\tb{\widetilde{B}}

\def\<{\langle}
\def\>{\rangle}
\def\ta{\widetilde{A}}
\def\tg{\mathcal{G}}

\def\pzab{\frac{\partial^2}{\partial z_1 \partial z_2}}

\def\pzz{\frac{\partial}{\partial \overline{z}}}

\def\px{\frac{\partial}{\partial x}}
\def\py{\frac{\partial}{\partial y}}
\def\tf{\tilde{f}}
\def\ea{e_{0}(\lambda)}
\def\ud{\underline}
\newcommand{\dd}{\mathrm{d}}
\newcommand{\ii}{\mathrm{i}}
\def\tbhat{\tb^{\<i\>}} 
\def\tbbra{\tb^{(i)}} 
\def\gbra{G^{(i)}} 
\def\fbra{F^{(i)}} 
\def\ghat{G^{\<i\>}} 
\def\fhat{F^{\<i\>}} 
\def\IE{\mathbb{I}\mkern-2mu\mathbb{E}_{\gi}}

\newcommand{\wt}{\widetilde}
\newcommand{\li}{{\langle i\rangle}}
\newcommand{\bi}{{\lbrace i \rbrace}}
\newcommand{\rir}{{[i]}}

\allowdisplaybreaks

\marginsize{25mm}{25mm}{25mm}{26mm}

\begin{document}
\begin{minipage}{0.85\textwidth}
 \vspace{2.5cm}
 \end{minipage}
\begin{center}
\large\bf
Central limit theorem for mesoscopic eigenvalue statistics\\ of the free sum of matrices
\end{center}

\renewcommand{\thefootnote}{\fnsymbol{footnote}}	
\vspace{1cm}

\begin{center}
 \begin{minipage}{1.1\textwidth}

\hspace{-1.5cm}  \begin{minipage}{0.33\textwidth}
\begin{center}
Zhigang Bao\footnotemark[1]  \\
		\footnotesize 
		{HKUST}\\
{\it mazgbao@ust.hk}
\end{center}
\end{minipage}
\begin{minipage}{0.33\textwidth}
	\begin{center}
		Kevin Schnelli\footnotemark[3]\\
		\footnotesize 
		{KTH Royal Institute of Technology}\\
		{\it schnelli@kth.se}
	\end{center}
\end{minipage}
\begin{minipage}{0.33\textwidth}
\begin{center}
Yuanyuan Xu\footnotemark[2]\\
\footnotesize 
		{KTH Royal Institute of Technology}\\
{\it yuax@kth.se}
\end{center}
\end{minipage}
\end{minipage}
\end{center}

\bigskip

\footnotetext[1]{Supported by the Hong Kong Research Grants Council ECS 26301517, GRF 16300618, and NSFC 11871425.}
\footnotetext[2]{Supported by the Swedish Research Council Grant VR-2017-05195.}
\footnotetext[3]{Supported by  the G\"oran Gustafsson Foundation and the Swedish Research Council Grant VR-2017-05195.}

\renewcommand{\thefootnote}{\fnsymbol{footnote}}	

\vspace{1cm}

\begin{center}
 \begin{minipage}{0.83\textwidth}\footnotesize{
 {\bf Abstract.}
We consider random matrices of the form $H_N=A_N+U_N B_N U^*_N$, where $A_N$, $B_N$ are two $N$ by $N$ deterministic Hermitian matrices and $U_N$ is a Haar distributed random unitary matrix. We establish a universal Central Limit Theorem for the linear eigenvalue statistics of~$H_N$ on all mesoscopic scales inside the regular bulk of the spectrum. The proof is based on studying the characteristic function of the linear eigenvalue statistics, and consists of two main steps: (1) generating Ward identities using the  left-translation-invariance of the Haar measure, along with a local law for the resolvent of $H_N$ and analytic subordination properties of the free additive convolution, allow us to derive an explicit formula for the derivative of the characteristic function; (2) a local law for two-point product functions of resolvents is derived using a partial randomness decomposition of the Haar measure. We also prove the corresponding results for orthogonal conjugations.}
\end{minipage}
\end{center}

 \vspace{5mm}

\thispagestyle{headings}

\section{Introduction}\label{sec_introduction}
In a seminal work Voiculescu~\cite{free_Vol} showed that two large Hermitian matrices are
{\it asymptotically free} if their eigenvectors are in general relative position. In particular, asymptotic freeness identifies the law of the sum of such large Hermitian matrices in terms of their respective spectra. A fundamental mechanism to generate asymptotic freeness is conjugation by independent unitary matrices that are distributed according to Haar measure. To be more specific, if $A_N$ and $B_N$ are two sequences of (deterministic and uniformly bounded) Hermitian matrices, and $U_N$ is a sequence of  Haar unitaries, then $A_N$ and $U_NB_NU_N^*$ are asymptotically free and the eigenvalue distribution of the {\it free sum} $H_N:=A_N+U_NB_NU_N^*$ is given by the {\it free additive convolution}, $\mu_A\boxplus\mu_B$, of the eigenvalue distributions $\mu_A$ of $A_N$, respectively $\mu_B$ of $B_N$, for large $N$.

One way of rephrasing this result is a law of large numbers: For sufficiently regular test functions $g$,
\begin{align}\label{le global law}
 \frac{1}{N}\sum_{i=1}^N g(\lambda_i)-\int_\R g(x)\dd \mu_A\boxplus\mu_B(x)
\end{align}
converges in probability to zero, as $N\rightarrow\infty$.

Having identified the free additive convolution as the limiting eigenvalue distribution, it is a natural question to consider fluctuations of such {\it linear eigenvalue statistics}. The theory of second order freeness developed by Collins, Mingo, \'Sniady, Speicher~\cite{second order 3,second order 1, second order 2} shows for analytic test functions $g$ that
\begin{align}\label{le first clt}
 \sum_{i=1}^N g(\lambda_i)-\sum_{i=1}^N \E g(\lambda_i)
\end{align}
converges in distribution to a centered Gaussian random variable, whose variance depends in an intricate way on the free additive convolution measure; see~\eqref{vf_global} below for an explicit expression for the variance. Using an analytic approach based on resolvent and characteristic function techniques, similar results were obtained by Pastur and Shcherbina in~\cite{random_book}. Conspicuously different from standard central limit theorem (CLT), the linear statistics in~\eqref{le first clt} is not rescaled by $N^{-1/2}$, which is explained by the strong correlations among the eigenvalues. 

In the present paper we are interested in the mesoscopic linear eigenvalue statistics for the free sum of matrices. We choose an energy $E$ inside the support of the free additive convolution measure, a test function $g \in C_c^2(\R)$ and consider the statistics
\begin{align}\label{observable}
 \sum_{i=1}^N g\Big(\frac{\lambda_i-E}{\eta}\Big)-\sum_{i=1}^N \E g\Big(\frac{\lambda_i-E}{\eta}\Big)\,,
\end{align}
where $\eta$ is an $N$-dependent spectral scale. The mesoscopic regime ranges over $N^{-1}\ll \eta\ll 1$, where the sum in~\eqref{observable} includes order $N\eta$ eigenvalues. For $\eta\sim 1$ the random variable~\eqref{observable} agrees with the macroscopic or global observable in~\eqref{le first clt}, while for $\eta\sim N^{-1}$ the sum in~\eqref{observable} is governed by single eigenvalues, where the statistics is determined by Dyson's sine kernel; see~\cite{CL}. 

Our main result shows that in the bulk spectrum the mesoscopic linear statistic~\eqref{observable} converges to a centered Gaussian random variable with variance given by
\begin{align}
 \frac{1}{4 \pi^2} \int_{\R}  \int_{\R}  \frac{(g(x_1)-g(x_2))^2}{(x_1-x_2)^2} \dd x_1 \dd x_2=\frac{1}{2 \pi} \int_{\R} |\xi| |\widehat{g}(\xi)|^2 \dd \xi\,,
\end{align}
which is the universal variance found in many other random matrices models, e.g.\ classical compact groups \cite{Soshnikov}, and invariant ensembles \cite{meso1,FKS}.

Our approach is based on an analysis of the characteristic function of~\eqref{observable} and establishes its convergence to the characteristic function of the limiting Gaussian distribution \cite{character,lytova+pastur}. Our proof has three main ingredients: Invariance properties of the Haar measure in the form of so-called Ward identities; analytic subordination for the free convolution measure; and local laws for two-point product functions of the resolvent.

Ward identities are used to compute the derivative of the characteristic function and will allow us to connect the variance of fluctuations to the analytic subordination phenomenon of free probability. The Stieltjes transform of the free convolution measure can be described by an analytic change of variables from the Stieltjes transforms of the measures $\mu_A$ or $\mu_B$. This is referred to as analytic subordination~\cite{Bia98,Voi93}, and, in fact, may be used to give an analytic definition of the free additive convolution~\cite{belin_H,CG}; see Theorem~\ref{subordination_function} below. The subordination phenomenon carries over to the random matrix model $H_N$. The Green function or resolvent of $H_N$ is determined not only on global scales~\cite{invariant} but also on local scales just above the microsopic scale by the Stieltjes transform of $\mu_A\boxplus\mu_B$. Such local laws giving strong rigidity estimates for the eigenvalues were established in~\cite{eta1,eta2} down to the optimal scale, see also~\cite{stability,concentrate} for previous results on some mesoscopic scales. Local laws also yield optimal speed of convergence estimates for~\eqref{le global law} inside the bulk spectrum and at regular spectral edges~\cite{edge}. In our proof we use stability properties for the subordination equations established in~\cite{stability} and the local laws of~\cite{eta2} to bound various error terms. A main technical difficulty in this paper is to derive systems of self-consistent equations for two-point product functions of resolvents appearing in the variance term for the linear statistics. We rely on a partial randomness decomposition of the Haar measure that was previously used to derive local laws for the resolvent in~\cite{eta1}. This technique allows us to exploit fluctuations on all mesoscopic scales and surpasses more conventional approaches where concentration with respect to the full Haar measure and Ward identities are used.

Mesoscopic linear statistics were studied for Wigner matrices~\cite{meso2, moment,character, mesowigner} and many other random matrix models such as the orthogonal polynomial ensembles \cite{Breuer+Duits}, Dyson Brownian motion \cite{Duits+Johansson,character2}, invariant $\beta$-ensembles \cite{Bekerman+lodhia,BEYY,lambert_1} and random band matrix \cite{erdos+knowles,erdos+knowles2}. Mesoscopic linear statistics are not 
only interesting in their own right, they are also found applications in the theory of homogenization for Dyson's Brownian motion (DBM) introduced by Bourgade, Erd\H{o}s, Yau and Yin~\cite{BEYY} to prove fixed energy universality of the local eigenvalue statistics of Wigner matrices.  Landon, Sosoe and Yau \cite{character2} subsequently derived a mesoscopic CLT to show fixed energy universality of the DBM. Mesoscopic central limit theorems combined with DBM were also used in~\cite{Bourgadegap,logarithm,character} to derive Gaussian fluctuations of single eigenvalues.

This paper is organized as follows. In Section~\ref{sec_main_results}, we introduce the model in more detail and state our main results. We also give an outline of the proof in Subsection~\ref{subsection outline of proof}. We collect some preliminary results, e.g., local stability of the subordination equations and local laws for the Green function, in Section~\ref{sec_preliminaries}. In Sections~\ref{sec_proof_main_theorem} and~\ref{sec_proof_main_lemma} the main arguments of the proofs are given. In the short Section~\ref{sec_expectation} we complement the results by computing the so-called bias. All our methods carry over to the orthogonal setup where one of the matrices is conjugated by Haar orthogonal matrices. The orthogonal case is analyzed separately in Section~\ref{sec_orthogonal}. The proofs of some technical results used in Sections~\ref{sec_proof_main_theorem}--\ref{sec_orthogonal} are postponed to the Appendix.

We conclude this introductory section by collecting some notational conventions used throughout the paper.
We use the following notion for high-probability estimates: 
\begin{definition}\label{definition of stochastic domination}
Let $\mathcal{X}\equiv \mathcal{X}^{(N)}$ and $\mathcal{Y}\equiv \mathcal{Y}^{(N)}$ be two sequences of
 nonnegative random variables. We say~$\mathcal{Y}$ stochastically dominates~$\mathcal{X}$ if, for all (small) $\epsilon>0$ and (large)~$D>0$,
\begin{align}
\mathbb{P}\big(\mathcal{X}^{(N)}>N^{\epsilon} \mathcal{Y}^{(N)}\big)\le N^{-D},
\end{align}
for sufficiently large $N\ge N_0(\epsilon,D)$, and we write $\mathcal{X} \prec \mathcal{Y}$ or $\mathcal{X}=O_\prec(\mathcal{Y})$.
\end{definition}

For any vector ${\bm y} \in \C^N$, denoted by bold font, we use $\|{\bm y}\|_2$ to denote the Euclidean norm. We write ${\bm g}=(g_i)_{i=1}^N \sim N_{\R}(0, \sigma^2 I_N)$ if $g_1, \cdots, g_N$ are i.i.d.\ centered Gaussian random variables $N(0,\sigma^2)$. In the complex case, ${\bm g} \sim N_{\C}(0, \sigma^2 I_N)$ means that $\Re g_i$ and $\Im g_i$ are i.i.d.\ Gaussian random variables $N(0,\frac{1}{2}\sigma^2)$.

For a general random variable $\mathcal{X}$, we denote by
\begin{equation}\label{notation_E}
\<\mathcal{X}\>:=\mathcal{X}-\E[\mathcal{X}]
\end{equation}
its centering. 
 
 For a matrix $X \in \C^{N \times N}$, we denote by $\|X\|_{\mathrm{op}}$ its operator norm and by $\|X\|_{\mathrm{HS}}$ its Hilbert-Schmidt norm. Moreover we use the convention $|X|^2=X^* X$. The normalized trace of $X$ is denoted by
\begin{equation}\label{notation_X}
\ud{X}:=\frac{1}{N} \Tr X\,.
\end{equation}

Finally, we use~$c$ and~$C$ to denote strictly positive constants that are independent of $N$. Their values may change from line to line. We write $X \ll Y$ if there is small $\epsilon>0$ such that $|X|\le N^{-\epsilon}|Y|$ as $N \rightarrow \infty$. We write $X=O(Y) $ if there exists a  constant $C>0$ such that $|X| \leq C |Y|$. We write $X \sim Y$  if there exist constants $c, C>0$ such that $c |Y| \leq |X| \leq C |Y|$. We denote the complex upper half-plane by $\C^+:=\{z\in\C\,:\,\im z>0\}$.

\section{Main results}\label{sec_main_results}

\subsection{Setup}
Consider a sequence of random Hermitian matrices of the form
\begin{equation}\label{H_N}
 H\equiv H_N= A+U B U^*\,,
\end{equation}
where $A \equiv A_N=\mathrm{diag}(a_i)$ and $B \equiv B_N=\mathrm{diag}(b_i)$ are two sequences of $N$ by $N$ deterministic real diagonal matrices, and $U \equiv U_N$ are $N$ by $N$ random unitary matrices distributed according to the Haar measure on the unitary group of order $N$, $U(N)$. Without loss of generality, by shifting with multiples of the identity matrix, we may assume that $\Tr A=\Tr B=0.$

We assume that for a constant $M$, independent of $N$,
\begin{equation}\label{assumption_1}
\sup_{N}\|A\|_{\mathrm{op}} \leq M\,,\quad \sup_N \|B\|_{\mathrm{op}} \leq M\,.
\end{equation}
 The eigenvalues of $H$ are denoted by $(\lambda_i)_{i=1}^N$ in non-decreasing order. The empirical spectral measures of $A$, $B$ and $H$ are denoted by $\mu_A$, $\mu_B$ and $\mu_N$ respectively, i.e.,
$$\mu_A:=\frac{1}{N} \sum_{i=1}^N \delta_{a_i}; \qquad \mu_B:=\frac{1}{N} \sum_{i=1}^N \delta_{b_i}; \qquad \mu_N:=\frac{1}{N} \sum_{i=1}^N \delta_{\lambda_i}.$$ 

We will assume that $\mu_A$ and $\mu_B$ have weak limits as $N$ tends to infinity:
\begin{assumption}\label{assumption_2}
There are deterministic compactly supported Borel probability measures $\mu_{\alpha}$ and $\mu_{\beta}$ on $\R$, neither of them being a single point mass and at least one of them being supported at more than two points, such that $\mu_A$ and $\mu_B$ converge weakly to $\mu_{\alpha}$ and $\mu_{\beta}$, respectively,  as $N \rightarrow \infty$. More precisely, we assume that
\begin{align}\label{levy distance assumption}
d_{\mathrm{L}}(\mu_A, \mu_\alpha) +d_{\mathrm{L}}(\mu_B, \mu_\beta) \rightarrow 0, \qquad N \rightarrow \infty,
\end{align}
where $d_{\mathrm{L}}$ denotes the L\'evy distance.
\end{assumption}
Bercovici and Voiculescu~\cite{BeV93} showed that the free additive convolution is continuous
with respect to weak convergence of measures. More specifically~\eqref{levy distance assumption} implies
\begin{align}\label{le difference levy distance}
d_{\mathrm{L}}(\mu_A \boxplus \mu_B, \mu_\alpha \boxplus \mu_\beta) \leq  d_{\mathrm{L}}(\mu_A, \mu_\alpha)+d_{\mathrm{L}}(\mu_B,\mu_\beta)\,.
\end{align}
The assumption that neither of $\mu_\alpha$, $\mu_\beta$ is a single point mass excludes trivial shifts by multiples of identities. The additional condition in Assumption~\ref{assumption_2} that at least one of them is supported at more than two points is related to the stability of the subordination equations and the arguments of Section~\ref{sec_proof_main_lemma}. Yet, the special case when $\mu_\alpha$ and $\mu_\beta$ are both two-point masses can be treated by combining our methods and results in Section 7 of \cite{stability} and Appendix B of \cite{eta1}.

We next recall the analytic definition of the free additive convolution. For a probability measure $\mu$ on $\R$ denote by $m_\mu$ its {\it Stieltjes transform}, i.e.
\begin{align}\label{stieltjes}
 m_\mu(z):=\int_\R\frac{\dd\mu(x)}{x-z}\,,\qquad z\in\C^+\,.
\end{align}
Note that $m_{\mu}\,:\C^+\rightarrow\C^+$ is analytic and can be analytically extended to the real line outside the support of $\mu$. Moreover, $m_{\mu}$ satisfies 
\begin{equation}\label{m_condition}
\lim_{\eta\nearrow\infty}\ii\eta {m_{\mu}}(\ii \eta)=-1.
\end{equation}
Conversely, if $m\,:\,\C^+\rightarrow\C^+$ is analytic and satisfies $\lim_{\eta\nearrow\infty}\ii\eta m(\ii\eta)=-1$, then $m$ is the Stieltjes transform of a probability measure $\mu$, i.e., $m(z) = m_{\mu}(z)$, for all $z\in\C^+$; see e.g.~\cite{Ak}. For notational simplicity we further introduce the  {\it negative reciprocal Stieltjes transform} of $\mu$ by setting
\begin{equation}\label{stieltjes_inverse}
F_{\mu}(z):=-\frac{1}{m_{\mu}(z)}.
\end{equation}
Note that $F_{\mu}: \C^+ \rightarrow \C^+$ is analytic and satisfies
\begin{equation}\label{F_condition}
\lim_{\eta\nearrow\infty} \frac{F_{\mu}(\ii\eta)}{\ii\eta} =1.
\end{equation}

The free additive convolution of two probability measures on the real line is characterized by the following result.
\begin{theorem}\label{subordination_function}
Given any Borel probability measures $\mu_{\alpha}$ and $\mu_{\beta}$ on $\R$, there exist unique analytic functions, $\omega_{\alpha}, ~\omega_\beta: \C^+ \rightarrow \C^+$ such that
\begin{enumerate}
\item for all $z \in \C^+$, $\Im\omega_\alpha(z), \Im \omega_\beta(z) \geq \Im z$, and
\begin{equation}\label{omega_far}
\lim_{\eta\nearrow\infty} \frac{\omega_{\alpha}(\ii\eta)}{\ii\eta} =\lim_{\eta\nearrow\infty} \frac{\omega_{\beta}(\ii\eta)}{\ii\eta}=1;
\end{equation}
\item for all $z \in \C^+$
\begin{equation}\label{omega_limit}
F_{\mu_\alpha}(\omega_\beta(z))=F_{\mu_\beta}(\omega_\alpha(z)); \qquad \omega_{\alpha}(z)+\omega_{\beta}(z)-z=F_{\mu_\alpha}(\omega_\beta(z)).
\end{equation}
\end{enumerate}
\end{theorem}
Hence, by (\ref{omega_far}) the function 
$${F}_{\mu_\alpha\boxplus \mu_\beta}(z):=F_{\mu_\alpha}(\omega_\beta(z))=F_{\mu_\beta}(\omega_\alpha(z))$$
 satisfies (\ref{F_condition}) and thus is the negative reciprocal Stieltjes transform of a probability measure, the free additive convolution of $\mu_\alpha$ and $\mu_\beta$. The functions $\omega_\alpha$ and $\omega_\beta$ are referred to as {\it subordination functions}. It was shown by Belinschi \cite{belin_1, belin_2} that if both $\mu_\alpha$ and $\mu_\beta$ are compactly supported probability measures on $\R$ and are supported on more than one point, then ${F}_{\mu_\alpha\boxplus \mu_\beta}$, $\omega_{\alpha}$ and $\omega_\beta: \C^+ \rightarrow \C^+$ can be extended continuously to $\R$. The singular continuous part of $\mu_{\alpha} \boxplus \mu_{\beta}$ is always zero while the absolutely continuous part is always non-zero. The corresponding density, denoted by $\rho_{\mu_\alpha \boxplus \mu_\beta}$, is real analytic whenever positive and finite. Atoms in the free additive convolution measure are identified as follows~\cite{BeV93}. A point $c\in\R$ is an atom of $\mu_\alpha\boxplus\mu_\beta$, if and only if there exist $a,b\in\R$ such that $c=a+b$ and $\mu_\alpha(\{a\})+\mu_\beta(\{b\}) > 1$. In fact, it was shown in~\cite{belin_3} that the density of $\mu_\alpha\boxplus\mu_\beta$ is always bounded if $\mu_\alpha(\{a\})+\mu_\beta(\{b\})<1$, for all $a,b\in\R$.

Returning to the free sum of matrices, we first consider the linear eigenvalue statistics in~\eqref{le first clt} on the global scale. Gaussian fluctuations of the linear eigenvalue statistics for analytic test functions were derived in \cite{second order 3} within the framework of second order freeness (we refer also to the monograph~\cite{free_prob_book}) and in~\cite{random_book} using resolvent based methods.
\begin{theorem}[Theorem 10.2.6 \cite{random_book}]
Let $H_N$ be of the form (\ref{H_N}) and satisfy (\ref{assumption_1}) and (\ref{levy distance assumption}). Let $g \in C(\R)$  be analytic in a neighborhood of $[-2M,2M]$, where $M$ is the constant in~\eqref{assumption_1}. Then the linear eigenvalue statistics,
\begin{equation}\label{linear_stat}
\Tr g(H_N)-\E \Tr g(H_N),
\end{equation}
 converges in distribution to a centered Gaussian random variable of variance
\begin{equation}\label{vf_global}
-\frac{1}{4 \pi^2} \int_{\mathcal{C}_2} \int_{\mathcal{C}_1} g(z_1) g(z_2) S(z_1,z_2) \dd z_1 \dd z_2,
\end{equation}
where the kernel $S(z_1,z_2)$ is given by
$$S(z_1,z_2)=\pzab \log \Big( \frac{(\omega_\alpha(z_1)-\omega_\alpha(z_2)) (\omega_\beta(z_1)-\omega_\beta(z_2)) }{(z_1-z_2)({{F}_{\mu_\alpha\boxplus \mu_\beta}(z_1)-{F}_{\mu_\alpha\boxplus \mu_\beta}(z_2)})} \Big),$$
and where $\mathcal{C}_{1,2}$ are contours enclosing $[-2M,2M]$ and are lying in the domain of analyticity of $g$.
\end{theorem}

In the present paper, we prove Gaussian fluctuations for the linear eigenvalue statistics (\ref{observable}) on the mesoscopic scales and establish a universal mesoscopic CLT inside the {\it regular bulk}: The regular bulk of $\mu_{\alpha}\boxplus \mu_{\beta}$ is the open set on which $\mu_\alpha \boxplus \mu_\beta$ has a continuous density that is strictly positive and bounded from above, i.e.,
\begin{equation}\label{omega}
\mathcal{B}_{\mu_\alpha \boxplus \mu_\beta}:=\Big\{ x \in \mbox{supp}(\rho_{\mu_\alpha \boxplus \mu_\beta})\,:\, \rho_{\mu_\alpha \boxplus \mu_\beta} (x)> 0\,; \quad \lim_{\eta \searrow 0} F_{\mu_\alpha \boxplus \mu_\beta}(x+\ii \eta) \neq 0  \Big\}.
\end{equation}
By the remarks after Theorem~\ref{subordination_function}, the regular bulk is always non-empty under Assumption~\ref{assumption_2}. 

The convergence rate in~\eqref{levy distance assumption} of Assumption \ref{assumption_2} may be very slow. Yet, by working with the finite-N deterministic measures $\mu_A \boxplus \mu_B$ instead of $\mu_\alpha \boxplus \mu_\beta$ we avoid issues related to this. Theorem \ref{subordination_function} ensures that there exist unique analytic functions, $\omega_{A}, ~\omega_B: \C^+ \rightarrow \C^+$ such that 
\begin{equation}\label{omega_finite}
F_{\mu_A}(\omega_B(z))=F_{\mu_B}(\omega_A(z)); \qquad \omega_{A}(z)+\omega_{B}(z)-z=F_{\mu_A}(\omega_B(z)),
\end{equation}
and 
\begin{equation}\label{F_definition}
F_{\mu_A\boxplus\mu_B}(z):=F_{\mu_A}(\omega_B(z))=F_{\mu_B}(\omega_A(z))
\end{equation} 
is the negative reciprocal Stieltjes transform of the free additive convolution of $\mu_A$ and $\mu_B$. 

Besides the unitary conjugation in (\ref{H_N}), we also consider orthogonal conjugations, i.e., the matrix
\begin{equation}\label{beta}
H=A+O B O^T\,, 
\end{equation}
where $O \equiv O_N$ is Haar distributed on the orthogonal group $O(N)$ and obtain the corresponding results. We will use the conventional symmetry parameter $\bm{\beta}$ as indicator for the symmetry class; $\bm\beta=2$ for unitary and $\bm\beta=1$ for orthogonal conjugations.

\subsection{Main results}
Choose a nonempty compact interval $\mathcal I$ within the regular bulk $\mathcal{B}_{\mu_\alpha \boxplus \mu_\beta}$ (see (\ref{omega})) and fix $E_0 \in \mathcal{I}$. Choose an $N$-dependent $\eta_0$ such that $ N^{-1} \ll \eta_0 \ll 1$. We then consider a mesoscopic test function 
\begin{equation}\label{fn}
f(x) \equiv f_{N}(x):=g\Big( \frac{x-E_0}{\eta_0}\Big), \qquad g \in C_c^2(\R), \qquad x \in \R.
\end{equation}
Following \cite{character,lytova+pastur,random_book}, we study the characteristic function 
\begin{equation}\label{mme}
\phi(\lambda):=\E[e(\lambda)], \quad \mbox{where}~e(\lambda):=\exp \Big\{ \i \lambda(\Tr f(H_N)-\E \Tr f(H_N)) \Big\}, \qquad \lambda \in \R.
\end{equation}
We have the following result for the characteristic function $\phi$.
\begin{proposition}\label{prop}
	Let $H_N$ be of the form (\ref{H_N}), satisfying (\ref{assumption_1}) and Assumption \ref{assumption_2}. Let $N^{-1+c_0}\leq \eta_0\leq N^{-c_0}$, for some small $c_0>0$. Assume in addition that there is a small $c>0$, such that $ |m'_{\mu_\alpha \boxplus \mu_\beta}(E_0+ \ii 0)|>c$. Then there exists $0 < \tau < c_0/6$, such that the characteristic function $\phi$ satisfies
		\begin{align}\label{le phi prime}\phi'(\lambda)=-\lambda \phi(\lambda) V(f)+\tilde{\mathcal{E}}, 
		\end{align}
		where
	\begin{align}\label{vf}
	 V(f)=-\frac{1}{2 \bm\beta \pi^2} \int_{\Gamma_{1}} \int_{\Gamma_{2}}  \tf(z_1)  \tf(z_2) \mathcal K(z_1,z_2) \dd z_1 \dd z_2,
	\end{align}
	$\tilde{\mathcal{E}}$ is an error term, and $\bm\beta=1,2$ is the symmetry parameter. The integral kernel $\mathcal K$ in (\ref{vf}) is given by
	\begin{align}\label{kernel}
	\mathcal K(z_1,z_2)= \pzab \log \bigg( \frac{(\omega_A(z_1)-\omega_A(z_2)) (\omega_B(z_1)-\omega_B(z_2)) }{(z_1-z_2)({F_{\mu_A\boxplus\mu_B}(z_1)-F_{\mu_A\boxplus\mu_B}(z_2)})} \bigg);
	\end{align}
         the function $\tf$ is an almost analytic extension of $f$ given in Lemma \ref{helffler} below; the contours $\Gamma_1$, $\Gamma_2$ are $\Gamma_1=\{ z_1 \in \C\,:\, |\Im z_1| = N^{-\tau} \eta_0 \}$ and $\Gamma_2=\{ z_2 \in \C\,:\, |\Im z_2| = \frac{1}{2}N^{-\tau} \eta_0 \}$ with counterclockwise orientation.
	
	The error term $\tilde{\mathcal{E}}$ in~\eqref{le phi prime} is bounded as
	\begin{align}
	|\tilde{\mathcal{E}}|=O_{\prec}\bigg( |\lambda| (\log N) N^{- \tau} \bigg)+O_{\prec}\bigg( \frac{(1+|\lambda|)N^{3 \tau}}{\sqrt{N \eta_0 }}\bigg)\,, \label{20011501}
	\end{align}
	provided that $ V(f) \prec 1$. 
	\end{proposition}
	
	The condition~$|m'_{\mu_\alpha \boxplus \mu_\beta}(E_0+ \ii 0)|>0$ helps us to control in Propositions~\ref{variance_bfbg} and~\ref{gbg} some error terms effectively. It ensures that $m_{\mu_\alpha\boxplus\mu_\beta}$ is locally injective in a neighborhood of $E_0$, yet it may not be a necessary condition for the results to hold. The condition is satisfied for familiar distributions of random matrix theory such as Wigner's semicircle law or the Marchenko-Pastur law.

The expectation of $\Tr f(H_N)$ has the following asymptotic expansion for the so-called bias.

\begin{proposition}\label{prop2}
Under the same assumptions and notations as in Proposition~\ref{prop}, the bias is given by
	\begin{equation}\label{bf}
	\E \Tr f(H_N)-N \int_{\R} f(x) \dd \mu_A\boxplus\mu_B(x)= \frac{1}{2 \pi \ii} \int_{\Gamma_1} \tf(z) b(z) \dd z+O(N^{-\tau})+O_{\prec}\bigg( \frac{N^{2\tau}}{\sqrt{N \eta_0}}\bigg)\,,
	\end{equation}
	with 
	\begin{equation}\label{bz}
	b(z):=\frac{1}{2} \bigg( \frac{2}{\bm\beta} -1\bigg) \frac{\dd}{\dd z} \log \bigg( \frac{\omega'_A(z) \omega'_B(z)}{F'_{\mu_A\boxplus\mu_B}(z)}  \bigg)\,,
	\end{equation}
	where $\bm\beta=1,2$ denotes the symmetry parameter.
\end{proposition}

Proposition \ref{prop} and \ref{prop2} imply the following universal mesoscopic CLT in the regular bulk.

\begin{theorem}[Universal mesoscopic CLT in the regular bulk]\label{meso}
	Under the same assumptions as in Proposition~\ref{prop}, for any test function $g \in C^2_c(\R)$, the mesoscopic linear statistics 
	\begin{equation}\label{linear_stat}
	\sum_{i=1}^N g \bigg( \frac{\lambda_i-E_0}{\eta_0} \bigg)-N \int_{\R} g \bigg(\frac{x-E_0}{\eta_0} \bigg) \dd \mu_A\boxplus\mu_B(x)\,,
	\end{equation}
	converges in distribution to a centered Gaussian random variable of variance
	\begin{align}\label{bulk_variance}
	\frac{1}{2\bm\beta \pi^2} \int_{\R}  \int_{\R}  \frac{(g(x_1)-g(x_2))^2}{(x_1-x_2)^2} \dd x_1 \dd x_2=\frac{1}{\bm\beta \pi} \int_{\R} |\xi| |\widehat{g}(\xi)|^2 \dd \xi,
	\end{align}
	where $\widehat{g}(\xi):=(2 \pi)^{-1/2} \int_{\R} g(x) \mathrm{e}^{-\i \xi x} \dd x$, and $\bm\beta=1,2$ is the symmetry parameter. In particular, the bias vanishes inside the regular bulk on mesoscopic scales.
\end{theorem}
	
	\begin{remark}
	Propositions \ref{prop} and \ref{prop2} can be extended using the Gromov-Milman concentration estimate to the regular spectral edges, where the density of the free convolution measure shows a square root behavior, under the restriction $ {N^{-2/5}}\ll \eta_0 \ll 1$. As a consequence, Theorem~\ref{meso} holds true on these scales but the limiting Gaussian law becomes
	${N}_\R\Big(\Big(\frac{2}{\bm\beta}-1\Big)\frac{g(0)}{4}, \frac{1}{\bm\beta \pi} \int_{\R} |\xi| |\hat{h}(\xi)|^2 \dd \xi \Big)$, where $h(x)=g(\mp x^2)$. Variance and bias agree with the expressions found for the Gaussian unitary and orthogonal ensembles~\cite{basor+widom,Min+Chen}; see also~\cite{huang,Li+Schnelli+Xu}. However, the mesoscopic scale at the regular edges ranges over ${N^{-2/3}}\ll \eta_0\ll 1$. The extension of these results to the full mesoscopic range remains an open problem.
	\end{remark}

\subsection{Outline of proof}\label{subsection outline of proof}

In this subsection we give an outline of the proof which is essentially split into two parts carried out in Sections~\ref{sec_proof_main_theorem} and~\ref{sec_proof_main_lemma}. 
Let
\begin{align}\label{H_tilde_definition}
H:=A+U B U^*; \qquad \mathcal{H}:= U^* H U= U^* A U+B 
\end{align}
and denote their resolvents or Green functions by
\begin{align}\label{tilde quantities}
G(z):=(H-zI)^{-1}, \qquad \tg(z):=(\mathcal{H}-zI)^{-1}, \qquad z \in \C^+.
\end{align}
Note that the Stieltjes transform of the empirical spectral measure $\mu_{N}$ of $H$ is given by
\begin{align}\label{m_N_definition}
m_N(z) \equiv m_{\mu_N}(z)=\frac{1}{N}\sum_{i=1}^N \frac{1}{\lambda_i-z}=\frac{1}{N} \Tr G(z)=\frac{1}{N} \Tr \mathcal G(z).
\end{align}
To simplify the notation, we let
$$\tb:=U B U^*; \qquad \ta:=U^* A U.$$

In the first part, we study the characteristic function $\phi(\lambda)$ of the linear statistics $\Tr f(H_N)-\E \Tr f(H_N)$, see~\eqref{mme}. Using the Helffer-Sj\"ostrand formula, Lemma~\ref{helffler}, we link the derivative of $\phi(\lambda)$ to the resolvent of $H_N$ as follows
\begin{equation}\label{phimini}
\phi'(\lambda)=\i \E\Big[ e(\lambda) (\Tr f(H_N)-\E \Tr f(H_N)) \Big]=\frac{\i}{\pi}\int_{\C} \frac{\partial}{\partial  \overline{z_1}} \tf(z_1)   \E \Big[ e(\lambda)(\Tr G(z_1)-\E \Tr G(z_1)) \Big]  \dd^2z_1\,,
\end{equation}
where $\tf$ is a quasi-analytic continuation of $f$; see~\eqref{fw}. We further use the Helffer-Sj\"ostrand formula to rewrite $e(\lambda)$ as
\begin{align}\label{emini}
 e(\lambda)=\exp\Big\lbrace \frac{\ii\lambda}{\pi}\int_\C\frac{\partial}{\partial \overline{z_2}}\tf(z_2) (\Tr G(z_2)-\E \Tr G(z_2))\dd ^2 z_2\Big\rbrace\,.
\end{align}
The integration domains of the spectral parameters $z_1$ in~\eqref{phimini} and $z_2$ in~\eqref{emini} are the whole complex plane. Thanks to the mesoscopic scaling in the test function $f$, recall~\eqref{fn}, the contributions from local scales are negligible to the integral on the right sides of~\eqref{phimini} and~\eqref{emini}. More precisely, following~\cite{character}, contributions from spectral parameters with imaginary parts much smaller in absolute value than $\eta_0$ are negligible and we restrict the integration to the domains  $\Omega_1\ni z_1$ defined~\eqref{domain} and $\Omega_2\ni z_2$ defined in~\eqref{domain2}. Moreover, we can replace $e(\lambda)$ by the regularized quantity $\ea$ of~\eqref{e22}. The details are presented in Subsection~\ref{subsection char}. We also mention that on the domains $\Omega_1$ and $\Omega_2$ we have optimal control of the resolvent $G(z)$ and its normalized trace $m_{N}(z)$ in terms of local laws in Theorem~\ref{local}, which will enable us to control various error terms.

\newcommand{\bs}{\boldsymbol}

\newcommand{\la}{\langle}
\newcommand{\ra}{\rangle}

From~\eqref{phimini} and~\eqref{emini} we are led to study the correlation
\begin{align}\label{le rock}
 \E[\ea (\Tr G(z_1)-\E \Tr G(z_1))]\,.
\end{align}
This is accomplished by using the left translation invariance of the Haar measure. Let $X=X^*$ be a deterministic $N$ by $N$ matrix and let $t\in\R$, then $\mathrm{e}^{\i t X}$ belongs to $U(N)$ and $U_t:= \mathrm{e}^{\i t X}U$ is also Haar distributed as $U$ by the translation invariance. Let $M\,:\,\C^{N \times N}\rightarrow \C$ be a differentiable map and introduce $H_t:= A+U_tBU_t^*$. Then  we have that $\E [M(H_t)]$ is constant in $t$ and thence 
\begin{align}\label{le how to get wardmini}
 \frac{{\rm d}}{{\rm d} t} \bigg|_{t=0}\E [M(H_t)]=0\,.
\end{align}
With different choices of functions $M$ and matrices $X$ we can generate identities among correlations functions of Green functions. In the physics literature such relations are often referred to as {\it Ward identities}. We can produce further Ward identities by considering the matrix $\mathcal{H}_t= U_t^*AU_t+B$ and proceed as above. We will treat the matrices $H=A+UBU^*$ and $\mathcal{H}=U^*AU+B$ in tandem, the deeper reason for this is that the subordination equations in~\eqref{omega_finite} form a two-by-two system. 

Combining different Ward identities with the subordination equations, we show in Subsection~\ref{subsection ward} that
\begin{align}\label{le rock 2}
 \E[\ea (\Tr G(z_1)-\E \Tr G(z_1)]&\approx\frac{ \ii \lambda}{\pi} \E \bigg[ \ea  \int_{ \Omega_2}   \frac{\partial }{\partial \overline{z_1}} \tf(z_2) \frac{\partial }{\partial {z_2}}K(z_1,z_2) \dd^2 z_2 \bigg]\,,
\end{align}
up to a negligible error term. The kernel $K$ is explicitly given in~\eqref{kernel_K_1}. It is a linear combination of the quantities
\begin{align}\label{le rock 3}
 K_{B,1}(z_1,z_2)&:=\frac{1}{N}\sum_{j=1}^N\frac{1}{a_j-\omega_B(z_1)}(\wt BG(z_2)G(z_1))_{ii}\,,\nonumber\\ K_{B, 2}(z_1,z_2)&:=\frac{1}{N}\sum_{j=1}^N\frac{1}{a_j-\omega_B(z_1)}(G(z_2)\wt B G(z_1))_{ii}\,,
\end{align}
and respective counterparts involving the matrix $A$ and Green functions of the matrix $\mathcal{H}$.  We remark at this point that we heavily relied on the local laws in Theorem~\ref{local} to control the error term in~\eqref{le rock 2}. Identity~\eqref{le rock 2} is the main outcome of the first step in the proof; see Lemma~\ref{lemma_useful}.

In the second step of the proof we study the quantities in~\eqref{le rock 3}. The first term in~\eqref{le rock 3}, $K_{B,1}(z_1,z_2)$, is easy to understand thanks to the following identity for resolvents,
\begin{align}\label{le rock 4}
 G(z_1)-G(z_2)=(z_1-z_2)G(z_1)G(z_2)
\end{align}
which reduces the first term in~\eqref{le rock 3} to a one-point function that can be well-understood by the local laws.  However the second term in~\eqref{le rock 3}, $K_{B,2}(z_1,z_2)$ is harder to understand as we cannot use cyclicity and the resolvent identity~\eqref{le rock 4} to reduce it to a one-point function. This term in fact constitutes one of the main technical difficulty of this paper. We have singled out the analysis in Section~\ref{sec_proof_main_lemma}. Recently, similar two-point product functions have been studied for ensembles with independent entries in~\cite{Cipolloni et all} for the Hermitization of non-Hermitian random matrices.

To analyze $K_{B,2}$ it is not enough to rely only on Ward identities and local laws for the resolvent. We derive a local law for the two-point quantities $K_{B,2}(z_1,z_2)$ and $K_{A,2}(z_1,z_2)$. One strategy for that is to use the Gromov-Milman concentration inequality (see e.g. Section 4.4. of~\cite{AGZ}) to estimate $K_{B,2}-\E K_{B,2}$. However, it turns out that $K_{B,2}$ is not sufficiently regular in $z_1$ and $z_2$ to obtain an effective estimate for all mesoscopic scales. (More precisely, for $\eta_0\gg N^{-1/2}$ this method works out.)

Instead, we follow the approach of~\cite{eta1}. It relies on a decomposition of Haar measure on the unitary groups given, e.g., in~\cite{DiaSha, Mezzadri}. For any fixed $1 \leq i\leq N$, any Haar unitary $U$ can be written as 
\begin{align}\label{le first decomposition}  
U=-\mathrm{e}^{\mathrm{i}\theta_i}R_i\,U^{\langle i\rangle}\,.
\end{align}
Here $R_i$ is the {\it Householder reflection} (up to a sign) sending the vector $\bs{e}_i$ to $\bs{v}_i$, where $\bs{v}_i\in\C^N$ is a random vector distributed uniformly on the complex unit $(N-1)$-sphere, and $\theta_i\in [0,2\pi)$ is the argument of the $i$th coordinate of $\bs{v}_i$. The unitary matrix $U^{\langle i\rangle}$ has $\bs{e}_i$ as its $i$th column and its $(i,i)$-matrix minor is Haar distributed on $U(N-1)$. The gist of the decomposition in~\eqref{le first decomposition} is that $R_i$ and the unitary $U^{\langle i\rangle}$ are independent, for each fixed $1 \leq i\leq N$. Hence, the decomposition in~\eqref{le first decomposition} allows one to split off the partial randomness of the vector~$\bs{v}_i$ from~$U$. 

The analysis of $K_{B,2}$ is split into two parts: For each index $i$, we establish a concentration estimate for $(G(z_2)\wt B G(z_1))_{ii}$ around the partial average $\E_{\bs{v}_i}[(G(z_2)\wt B G(z_1))_{ii}]:=\E[(G(z_2)\wt B G(z_1))_{ii}|U^{\langle i\rangle}]$. This concentration is stronger than in the conventional Gromov-Milman inequality, as we are integrating out order $N$ variables (the entries of $\bs{v}_i$) rather than order $N^2$ variables when taking the full expectation with respect to Haar measure. The details are given in Lemmas~\ref{remove} and~\ref{concentrate_2}.

In the second part, we identify~$\E_{\bs{v}_i}(G(z_2)\wt B G(z_1))_{ii}$. Using the decomposition~\eqref{le first decomposition} and the notation $\wt{B}^{\la i\ra}:= U^{\li}B(U^{\li})^*$, one works out, using concentration estimates with respect to $\bs{v}_i$, that
\begin{align}\label{le rock 5}
 \E_{\bs{v}_i}(G(z_2)\wt B G(z_1))_{ii}\approx-\E_{\bs{v}_i}\mathrm{e}^{\ii \theta_i}\bs{v}_i^* \wt{B}^{\la i\ra}G(z_2)\wt BG(z_1)\ei\,.
\end{align}
We then introduce the two-point product functions 
\begin{align}\label{le S2t}
 S_i^{[2]}(z_1,z_2):=\mathrm{e}^{\ii \theta_i}\bs{v}_i^* \wt{B}^{\la i\ra}G(z_2)\wt BG(z_1)\ei\,, \qquad T_i^{[2]}(z_1,z_2):=\mathrm{e}^{\ii \theta_i}\bs{v}_i^* G(z_2)\wt BG(z_1)\ei\,,
\end{align}
as well as the one-point functions
\begin{align}\label{le S1t}
S_i^{[1]}(z_1):=\mathrm{e}^{\ii \theta_i}\bs{v}_i^* \wt{B}^{\la i\ra} G(z_1)\ei\,, \qquad T_i^{[1]}(z_1):=\mathrm{e}^{\ii \theta_i}\bs{v}_i^* G(z_1)\ei\,, 
\end{align}
where the latter were already used in~\cite{eta1}. (The definitions of the quantities used in Section~\ref{sec_proof_main_lemma} are for technical reasons slightly different, but for simplicity we use here the versions above.)

Next, approximating $\mathrm{e}^{-\mathrm{i}\theta_i}\bs{v}_i$ by a Gaussian vector and using integration by parts in $\E_{\bs{v}_i}S_i^{[\sharp]}$ and $\E_{\bs{v}_i}T_i^{[\sharp]}$, with $\sharp=1,2$,  we obtain a system of equations linking $\E_{\bs{v}_i}S_i^{[2]}$ and $\E_{\bs{v}_i}T_i^{[2]}$, which can approximately be solved. This step involves local laws for the quantities $S_i^{[1]}(z)$ and $T_i^{[1]}(z)$ alongside with a further Ward identity (Lemma~\ref{mathcal_Y}) that were established in~\cite{eta1}. Interestingly, it suffices to monitor the four quantities in~\eqref{le S2t} and~\eqref{le S1t} to close that system and no higher order correlation functions appear. Once we have found an expression for $\E_{\bs{v}_i}S_i^{[2]}$, we can identify $(G(z_2)\wt B G(z_1))_{ii}$ via~\eqref{le rock 5}. In this argument, we require the condition $m_{fc}'(E_0+\ii\eta_0)\not=0$ of Theorem~\ref{meso} to control some error terms. The results of this analysis are summarized in Proposition~\ref{variance_bfbg}. This will conclude Section 5.

With Proposition~\ref{variance_bfbg}, we can compute the kernel $K(z_1,z_2)$ in~\eqref{le rock 2}. In Section~\ref{sec_proof_main_theorem} we then prove Proposition~\ref{prop} and Theorem~\ref{meso} based on this result.

Along the way, we require some deterministic stability estimates on the subordination functions and the Jacobian associated with the subordination equations~\eqref{omega_finite}. Those are all collected in Section~\ref{sec_preliminaries}. The computation of the bias in the mesoscopic bulk is done in Section~\ref{sec_expectation}. In Section~\ref{sec_orthogonal}, we extend the analysis to the orthogonal setting. Finally, some technical estimates, in particular related to the concentration estimates with respect the vectors $(\bs{v}_i)$, are postponed to the Appendix.

\section{Preliminaries}\label{sec_preliminaries}
In this section, we collect some preliminary results: stability estimates and local laws for the Green function. Recall (\ref{stieltjes}) and (\ref{stieltjes_inverse}). To simplify the notation, we introduce the shorthands
$$m_{A}(z):=m_{\mu_A}(z), \quad m_{B}(z):=m_{\mu_B}(z), \quad F_{A}(z):=-\frac{1}{m_{\mu_A}(z)}, \quad F_{B}(z):=-\frac{1}{m_{\mu_B}(z)},$$
we abbreviate ${\mu}_{fc}:=\mu_A \boxplus \mu_B$, and denote the corresponding Stieltjes transform, negative reciprocal Stieltjes transform and density by $m_{fc}$, $F_{fc}$, and $\rho_{fc}$ respectively. They are $N$ dependent but deterministic. In addition, we write $\widetilde{\mu}_{fc}=\mu_\alpha \boxplus \mu_\beta$ and use $\widetilde{m}_{fc}$, $\widetilde{F}_{fc}$ and $\widetilde \rho_{fc}$ to denote the corresponding limiting Stieltjes transform, negative reciprocal Stieltjes transform and density, as $N \rightarrow \infty$.

\subsection{Properties of the subordination functions}
Recall the regular bulk $\mathcal{B}_{\mu_\alpha \boxplus \mu_\beta}$ defined in (\ref{omega}). We introduce a corresponding domain for the spectral parameter $z$,
\begin{equation}\label{bulk_domain}
D_{bulk}:=\big\{ z=E +\i \eta: E \in \mathcal{I}, \; N^{-1+\epsilon} < \eta \leq 1 \big\},
\end{equation}
where $\mathcal{I} \subset \mathcal{B}_{\mu_\alpha \boxplus \mu_\beta}$ is a nonempty compact interval and $\epsilon>0$ is a small constant. 

It was shown in \cite{stability,kargin_annal_1} that 
\begin{equation}\label{difference_m}
\max_{z \in D_{bulk}}\Big( |\omega_A(z)-\omega_\alpha(z)|+ |\omega_B(z)-\omega_\beta(z)|+ |m_{fc}(z)-\widetilde m_{fc}(z)|\Big)\leq C \Big( d_{\mathrm{L}}(\mu_A, \mu_\alpha) +d_{\mathrm{L}}(\mu_B, \mu_\beta)\Big),
\end{equation}
for $N$ sufficiently large, which directly implies~\eqref{le difference levy distance} by (\ref{levy distance assumption}).

Thanks to these convergence results, the qualitative properties of $\omega_A(z)$, $\omega_B(z)$ and $m_{fc}(z)$ asymptotically agree with the limiting $\omega_\alpha(z)$, $\omega_\beta(z)$ and $\widetilde m_{fc}(z)$ respectively, and we obtain the following estimates:

\begin{lemma}[Lemma 3.2 - 3.4 in \cite{stability}]\label{bound_bulk}
Under Assumption \ref{assumption_2}, we have the following estimates.
\begin{enumerate}
\item There exists $C>0$ such that
\begin{equation}\label{m_bound}
|m_{fc}(z)| \leq C\,; \qquad   |\omega_A(z)| \leq C\,; \qquad |\omega_B(z)| \leq C\,,
\end{equation}
uniformly for $ z \in D_{bulk}$, for sufficient large $N$.
\item There exists $c>0$ such that
\begin{equation}\label{imaginary_bound}
 |\Im m_{fc}(z)| \geq c\,; \qquad  |\Im \omega_A(z)| \geq c\,; \qquad | \Im \omega_B(z)| \geq c\,,
 \end{equation}
uniformly for $ z \in D_{bulk}$, for sufficient large $N$.
\item There exist $c, C>0$ such that
\begin{equation}\label{delta_bound}
c \leq |1-(F'_{A}(\omega_B(z))-1)(F'_{B}(\omega_A(z))-1)|\leq C\,,
\end{equation}
uniformly for $ z \in D_{bulk}$, for sufficient large $N$.
\item There exists $C>0$ such that
\begin{equation}\label{omega_prime_bound}
|\omega'_A(z)| \leq C\,; \qquad |\omega'_B(z)| \leq C\,, \qquad |m'_{fc}(z)| \leq C\,,
\end{equation}
uniformly for $ z \in D_{bulk}$, for sufficient large $N$.
\end{enumerate}
\end{lemma}

Note that the quantity estimated in (\ref{delta_bound}), henceforth denoted
\begin{equation}\label{delta}
\Delta(z):=1-(F'_{A}(\omega_B(z))-1)(F'_{B}(\omega_A(z))-1)\,,
\end{equation}
is the Jacobian of the following linear system obtained by differentiating the subordination equations~(\ref{omega_finite}),
\begin{align}\label{le system for omega primes}
\begin{pmatrix}
1& 1- F'_{A}(\omega_B(z))\\
1-F'_{B}(\omega_A(z)) & 1
\end{pmatrix}
\begin{pmatrix}
\omega_A'(z)\\
\omega_B'(z)
\end{pmatrix}
=
\begin{pmatrix}
1\\
1
\end{pmatrix}\,,\qquad z\in \C^+\,.
\end{align}
Since the Jacobian $\Delta(z)$ is non-vanishing for $z \in D_{bulk}$, we hence get from~\eqref{le system for omega primes} that 
\begin{equation}\label{system}
\begin{pmatrix}
\omega_A'(z)\\
\omega_B'(z)
\end{pmatrix}
=\frac{1}{\Delta(z)}
\begin{pmatrix}
F'_{A}(\omega_B(z))\\
F'_{B}(\omega_A(z))
\end{pmatrix}
=\frac{1}{\Delta(z) m^2_{fc}(z)}
\begin{pmatrix}
m'_{A}(\omega_B(z))\\
m'_{B}(\omega_A(z))
\end{pmatrix},
\end{equation}
where 
\begin{equation}\label{m_A_m_B}
 m'_A(\omega_B(z))=\frac{1}{N}  \sum_{j=1}^N \frac{1}{(a_j-\omega_B(z))^2};\qquad  m'_B(\omega_A(z))=\frac{1}{N}  \sum_{j=1}^N \frac{1}{(b_j-\omega_A(z))^2}.
\end{equation}
In combination with the lower bounds in (\ref{imaginary_bound}) and (\ref{delta_bound}), the linear system (\ref{system}) yields that $\omega'_A(z), \omega'_B(z)$ are of constant order in the regular bulk, as stated in (\ref{omega_prime_bound}).
Furthermore, by the subordination equations~\eqref{omega_finite}, we have
\begin{equation}\label{self_m_fc}
m_{fc}(z)=m_A(\omega_B(z))=m_B(\omega_A(z)), \qquad \omega_A+\omega_B-z=F_{fc}(z)=-\frac{1}{m_{fc}(z)}.
\end{equation}
Differentiating (\ref{self_m_fc}) with respect to $z$, we find that
\begin{equation}\label{differential_relation}
m_{fc}'(z)=\omega'_A(z) m'_B(\omega_A(z))=\omega'_B(z) m'_A(\omega_B(z)), \qquad \omega_A'(z)+\omega_B'(z)-1=F_{fc}'(z)=\frac{m_{fc}'(z)}{m_{fc}(z)^2}.
\end{equation}
The first relation in (\ref{differential_relation}) implies that $m'_{fc}(z)$ is also of constant order as are $\omega'_A(z)$ and $\omega'_B(z)$. If $m_{fc}'(z)$ is not zero, neither are the factors $\omega_A'(z)$ and $\omega_B'(z)$. Thus
\begin{equation}\label{LMprime}
 m'_A(\omega_B(z))=\frac{m_{fc}'(z)}{\omega'_B(z)};\qquad  m'_B(\omega_A(z))=\frac{m_{fc}'(z)}{\omega'_A(z)}.
\end{equation}
In addition, we have the following lemma whose proof is postponed to the Appendix.
\begin{lemma}\label{difference}
Under Assumption \ref{assumption_2}, we have
\begin{align}
\max_{z \in D_{bulk}}\Big(  |\omega'_A(z)-\omega'_\alpha(z)|+ |\omega'_B(z)-\omega'_\beta(z)|+ |m'_{fc}(z)-\widetilde m'_{fc}(z)|\Big) \leq C \Big( d_{\mathrm{L}}(\mu_A, \mu_\alpha) +d_{\mathrm{L}}(\mu_B, \mu_\beta)\Big)\,,
\end{align}
for $N$ sufficiently large.
\end{lemma}

\subsection{Variance kernel $\mathcal{K}(z_1,z_2)$}\label{another_formula}
In this subsection, we define some functions in terms of the subordination functions for later purpose and then rewrite the kernel (\ref{kernel}) of the variance expression (\ref{vf}) of the linear statistics in a form without singularities. Generalizing (\ref{delta}) and (\ref{m_A_m_B}), we introduce the following functions of two spectral parameters $z_1,z_2 \in \C \setminus \R$, 
\begin{align}
L_A(z_1,z_2)&:=\frac{1}{N}  \sum_{j=1}^N \frac{1}{(b_j-\omega_A(z_1))(b_j-\omega_A(z_2))}\,;\nonumber\\
L_B(z_1,z_2)&:= \frac{1}{N}  \sum_{j=1}^N \frac{1}{(a_j-\omega_B(z_1))(a_j-\omega_B(z_2))}\,;\label{LM_def}
\end{align}
and 
\begin{align}
\Delta(z_1,z_2)&:=1-\Big(\frac{L_A(z_1,z_2)}{m_{fc}(z_1)m_{fc}(z_2)}-1\Big)\Big(\frac{L_B(z_1,z_2)}{m_{fc}(z_1)m_{fc}(z_2)}-1\Big)\,.\label{delta2}
\end{align}
Note that if $z_1=z_2=z$, then 
\begin{equation}\label{LM_def_zz}
L_A(z):=L_A(z,z)=m'_B(\omega_A(z)), \qquad L_B(z):=L_B(z,z)=m'_A(\omega_B(z)); \qquad \Delta(z)=\Delta(z,z).
\end{equation}
As an analogue of (\ref{differential_relation}), for $z_1\not= z_2$, we have from \eqref{self_m_fc} that
\begin{align*}
m_{fc}(z_1)-m_{fc}(z_2)=\frac{1}{N}  \sum_{j=1}^N \frac{\omega_B(z_1)-\omega_B(z_2)}{(a_j-\omega_B(z_1))(a_j-\omega_B(z_2))}=(\omega_B(z_1)-\omega_B(z_2) )L_B(z_1,z_2)\\
=\frac{1}{N}  \sum_{j=1}^N \frac{\omega_A(z_1)-\omega_A(z_2)}{(b_j-\omega_A(z_1))(b_j-\omega_A(z_2))}=(\omega_A(z_1)-\omega_A(z_2)) L_A(z_1,z_2).
\end{align*}
If we choose $z_1 \neq z_2$ such that the difference $m_{fc}(z_1) - m_{fc}(z_2)$ is nonzero, then by the above relation neither are  $\omega_A(z_1)-\omega_A(z_2)$ and $\omega_B(z_1) - \omega_B(z_2)$. Therefore, dividing these two factors on both sides, we have
\begin{equation}\label{LM}
L_A(z_1,z_2)=\frac{m_{fc}(z_1)-m_{fc}(z_2)}{\omega_A(z_1)-\omega_A(z_2)}; \qquad L_B(z_1,z_2)=\frac{m_{fc}(z_1)-m_{fc}(z_2)}{\omega_B(z_1)-\omega_B(z_2)}.
\end{equation}
Using (\ref{LM}), we obtain the analogue of (\ref{system}), i.e.,
\begin{equation}\label{system2}
\begin{pmatrix}
\omega_A(z_1)-\omega_A(z_2)\\
\omega_B(z_1)-\omega_B(z_2)
\end{pmatrix}
=\frac{z_1-z_2}{\Delta(z_1,z_2) m_{fc}(z_1)m_{fc}(z_2)}
\begin{pmatrix}
L_B(z_1,z_2)\\
L_A(z_1,z_2)
\end{pmatrix}.
\end{equation}
Therefore, the kernel $\mathcal{K}$ in (\ref{kernel}) of the variance expression can be written as 
$$\mathcal{K}(z_1,z_2)=\pzab \log \Big( \frac{(\omega_A(z_1)-\omega_A(z_2)) (\omega_B(z_1)-\omega_B(z_2)) }{(z_1-z_2)(F_{fc}(z_1)-F_{fc}(z_2))} \Big)=- \pzab \log (\Delta(z_1,z_2)).$$ 
The benefit of this form is to avoid singularities caused by $z_1=z_2$ or $m_{fc}(z_1) = m_{fc}(z_2)$, since $L_A(z_1,z_2)$, $L_B(z_1,z_2)$ as well as $\Delta(z_1,z_2)$ are well-defined functions for all $z_1,z_2 \in D_{bulk}$. 

Similarly, using (\ref{system}) and (\ref{LMprime}), we can rewrite $b(z)$ in the bias formula (\ref{bz}) as
$$b(z)=\frac{1}{2} \Big( \frac{2}{\bm\beta} -1\Big) \frac{\dd}{\dd z} \log \Big( \frac{\omega'_A(z) \omega'_B(z)m^2_{fc}(z)}{m'_{fc}(z)}\Big)=-\frac{1}{2} \Big( \frac{2}{\bm\beta} -1\Big) \frac{\dd}{\dd z} \log \Delta(z).$$

\subsection{Local law for the Green function}
We end this section by stating the local laws for the Green functions of $H$ and $\mathcal{H}$ in (\ref{tilde quantities}) and~\eqref{m_N_definition}. For this purpose, we introduce the deterministic control parameter
\begin{equation}\label{control}
\Psi \equiv \Psi(z):=\frac{1}{\sqrt{N |\eta|}}\,,\qquad \qquad z=E+\ii\eta\in\C \setminus \R\,.
\end{equation}

\begin{theorem}[Theorem 2.5 in \cite{eta1},  Theorem 2.4 in \cite{eta2}]\label{local}
Under Assumption \ref{assumption_2} and (\ref{assumption_1}), the following estimates
\begin{align*}
&\max_{i,j}\left|G_{ij}(z)-\frac{1}{a_i-\omega_B(z)} \delta_{ij} \right| \prec \Psi(z) ,\qquad & |m_{N}(z)-m_{fc}(z) | \prec \Psi^2(z),&\\
&\max_{i,j}\left|(\tb G(z))_{ij}-\frac{z-\omega_B(z)}{a_i-\omega_B(z)} \delta_{ij} \right| \prec \Psi(z),\qquad  &\Big|\frac{1}{N}\mathrm{Tr}(\tb G(z))-(z-\omega_B)m_{fc}(z)\Big| \prec \Psi^2(z),&
\end{align*}
hold uniformly for all $z\in D_{bulk}$.

Furthermore, for any deterministic and uniformly bounded $d_1, \cdots, d_N \in \C$, we have
\begin{equation}\label{average}
\Big| \frac{1}{N} \sum_{i=1}^N d_i\Big(G_{ii}(z)-\frac{1}{a_i-\omega_B(z)} \Big) \Big| \prec \Psi^2(z)\,,
\end{equation}
uniformly for all $z\in D_{bulk}$. The same estimates hold for the Green function $\mathcal{G}$ in (\ref{tilde quantities}) with the roles of $A$ and $B$ interchanged. 
\end{theorem}

\section{Proof of Proposition \ref{prop} and Theorem \ref{meso}}\label{sec_proof_main_theorem}

In this section, we give the proof of Proposition~\ref{prop} and Theorem~\ref{meso} for the unitary case $\bm\beta=2$. The orthogonal case $\bm\beta=1$ is proved similarly in Section \ref{sec_orthogonal}.

\subsection{Characteristic function and its derivative}\label{subsection char}
 The idea is to study the derivative of the characteristic function of the linear eigenvalue statistics and link it with the resolvent of $H_N$ via the Helffer-Sj\"ostrand calculus. Let $f$ be the mesoscopic test function introduced in~\eqref{fn}. Recall from~\eqref{mme} the characteristic function
\begin{equation}\label{mmeZG}
\phi(\lambda):=\E[e(\lambda)], \quad  \mbox{with}~e(\lambda):=\exp \big\{ \i \lambda(\Tr f(H_N)-\E \Tr f(H_N)) \big\}\,,\quad \qquad\lambda \in \R\,.
\end{equation}
The following lemma is a version of the well-known Helffer-Sj\"ostrand formula.
\begin{lemma}\label{helffler}(Helffer-Sj\"ostrand formula) 
	Let $\chi(y)$ be a smooth cutoff function with support in $[-2,2]$, with $\chi(y)=1$ for $|y| \leq 1$. Define the almost-analytic extension of $f$ by 
	\begin{align}\label{le almost analytic extension}
	\tilde{f}(x+\ii y):=(f(x)+\ii y f'(x)) \chi(y)\,.
	\end{align} 
	Then, for any $w\in\R$,
	\begin{align}\label{le helffer sjostrand formula}
	f(w)=\frac{1}{\pi} \int_{\C}\pzz \frac{ \tilde{f}(z)}{w-z} \dd^2z=\frac{1}{2 \pi} \int_{\R^2} \frac{\ii y f''(x) \chi(y)+\ii \big( f(x)+\ii y f'(x) \big) \chi'(y)}{w-x-\ii y} \dd x \dd y\,,
	\end{align}
	where $z=x+\ii y$, $\pzz=\frac{1}{2}(\px+\ii \py)$, and $\dd ^2z$ denotes Lebesgue measure on $\C$.
\end{lemma}

Using~\eqref{le almost analytic extension} and the definition of the Green function in~\eqref{tilde quantities}, the spectral calculus yields the following representation of the linear eigenvalue statistics,
\begin{equation}\label{fw}
\Tr f(H_N) -\E \Tr f(H_N) =\frac{1}{\pi} \int_{\C} \pzz \tf(z) (\Tr G(z)-\E \Tr G(z))\dd^2z\,.
\end{equation}
Then taking the derivative of the characteristic function $\phi(\lambda)$, we obtain
\begin{equation}\label{phi}
\phi'(\lambda)=\i \E\Big[  (\Tr f(H_N)-\E \Tr f(H_N)) e(\lambda)\Big]=\frac{\i}{\pi}\int_{\C} \pzz \tf(z)   \E \Big[ e(\lambda)(\Tr G(z)-\E \Tr G(z)) \Big]  \dd^2z.
\end{equation}
As an observation in \cite{character}, instead of working on the $\C$, we can remove the ultra-local, or sub-mesoscopic, scales and restrict the integration domain in~\eqref{phi} to 
\begin{equation}\label{domain}
\Omega_1 := \big\{ z_1:=x_1 + \i y_1 \in \C: |y_1| \geq N^{-\tau} \eta_0 \big\},
\end{equation}
with $\tau>0$ and $\eta_0$ as in Proposition~\ref{prop}, without effecting the mesoscopic linear eigenvalue statistics.  Indeed, using that $y_1 \rightarrow \Im m_N(z_1)y_1$ is increasing, we can extend the local law as follows:
\begin{equation}\label{mmp}
\left| \Tr G(x_1+\i y_1)-\E \Tr G(x_1+\i y_1) \right| = O_{\prec}\Big(\frac{1}{|y_1|}\Big),
\end{equation}
uniformly in $|y_1|>0$ and $x_1 \in \mathcal{I}$; see (\ref{bulk_domain}). In addition, due to (\ref{fn}), there exists some $C>0$ such that
\begin{equation}\label{assumpf}
\int_{\R} |f(x)| \dd x \leq C \eta_0; \qquad \int_{\R} |f'(x)| \dd x \leq C'; \qquad \int_{\R} |f''(x)| \dd x \leq \frac{C''}{\eta_0},
\end{equation}
and thus we have
\begin{equation}\label{fw2}
\Tr f(H_N) -\E \Tr f(H_N)=\frac{1}{ \pi}  \int_{\Omega_1}  \frac{\partial }{\partial \overline{z_1}} \tf(z_1) (\Tr(G(z_1))-\E \Tr G(z_1)) \dd^2 z_1+O_{\prec}(N^{-\tau}).
\end{equation}
Similarly, we remove the ultra-local scales in the integral domain in the expression of $e(\lambda)$ and define
\begin{equation}\label{e22}
\ea:=\exp\Big\{ \frac{\i \lambda}{ \pi}   \int_{\Omega_2}  \frac{\partial }{\partial \overline{z_2}} \tf(z_2) (\Tr(G(z_2))-\E \Tr G(z_2)) \dd^2 z_2  \Big\},
\end{equation}
where
\begin{equation}\label{domain2}
\Omega_2 := \Big\{ z_2:=x_2 + \i y_2 \in \C: |y_2| \geq \frac{1}{2} N^{-\tau} \eta_0 \Big\}\,.
\end{equation}
It is straightforward to check that $\ea$ approximates $e(\lambda)$ as
\begin{equation}\label{e2}
|e(\lambda)-\ea|=O_{\prec}\big( |\lambda| N^{ -\tau}  \big)\,.
\end{equation}
Summarizing the above, we have the following lemma. 
\begin{lemma} Under the assumptions of Proposition~\ref{prop}, we have the representation
\begin{equation}\label{newphi}
\phi'(\lambda)=\frac{\i}{\pi}\int_{\Omega_1} \frac{\partial }{\partial \overline{z_1}} \tf(z_1)   \E \Big[ \ea (\Tr(G(z_1))-\E \Tr G(z_1)) \Big]  \dd^2z_1+O_{\prec}\Big( |\lambda| (\log N) N^{ -\tau} \Big).
\end{equation}
\end{lemma}
In view of~\eqref{newphi} we are led to study 
\begin{align}\label{le semi expression}
\E [ \ea (\Tr(G(z_1))-\E \Tr G(z_1))]
\end{align}
in the next subsection.

\subsection{Invariance of Haar measure: Ward identities}\label{subsection ward}

In this subsection, we study~\eqref{le semi expression} further. For simplicity, we recall the shorthands in (\ref{notation_E}) and (\ref{notation_X}). With these notations, we write~\eqref{le semi expression} as $\E[\ea \< \underline{G(z_1)}\>]$. We have the following estimate for~\eqref{le semi expression}.
\begin{lemma}\label{lemma_useful}
Let $z_1=E_1+\ii\eta_1\in D_{bulk}$. Then
 \begin{align}\label{sum1bis}
\E [\ea \<\Tr G(z_1)\>]=&\frac{ \ii \lambda}{ \pi} \E \bigg[ \ea  \int_{ \Omega_2} \frac{\partial }{\partial \overline{z_2}} \tf(z_2) \frac{\partial }{\partial {z_2}}K(z_1,z_2) \dd^2 z_2 \bigg]+O_{\prec}\Big( \frac{1}{N \eta_1^2 }\Big),
\end{align}
where the kernel $K(z_1,z_2)$ is given by
\begin{equation}\label{kernel_K_1}
K(z_1,z_2):= \frac{\omega_B'(z_1)}{m_{fc}(z_1)} (K_{B_1}(z_1,z_2)-K_{B_2}(z_1,z_2))+\frac{\omega_A'(z_1)}{m_{fc}(z_1)} (K_{A_1}(z_1,z_2)-K_{A_2}(z_1,z_2))\,,
\end{equation}
with\small
\begin{align}
K_{B,1}(z_1,z_2)&:=\frac{1}{N} \sum_{j=1}^N \frac{1}{a_j-\omega_B(z_1)}  (\tb G(z_2) G(z_1))_{jj}; \quad  K_{B,2}(z_1,z_2):=\frac{1}{N} \sum_{j=1}^N \frac{1}{a_j-\omega_B(z_1)} (G(z_2) \tb G(z_1))_{jj};\label{K_B_formula}\\
K_{A,1}(z_1,z_2):&=\frac{1}{N} \sum_{j=1}^N \frac{1}{b_j-\omega_A(z_1)}  (\ta \mathcal{G}(z_2) \mathcal{G}(z_1))_{jj}; \quad  K_{A,2}(z_1,z_2):=\frac{1}{N} \sum_{j=1}^N \frac{1}{b_j-\omega_A(z_1)} (\mathcal G(z_2) \ta \mathcal G(z_1))_{jj}.\label{K_A_formula}
\end{align}\normalsize
\end{lemma}

\begin{proof} The left side of~\eqref{sum1bis} involves the full expectation with respect the Haar measure on $U(N)$. This suggest to make use of its left invariance: Let $X=X^*$ be any deterministic $N$ by $N$ matrix and let $t\in\R$, then $\mathrm{e}^{\i t X}$ belongs to $U(N)$ and $U_t:= \mathrm{e}^{\i t X}U$ is also Haar distributed as $U$ is by assumption.

Let $M\,:\,\C^{N \times N}\rightarrow \C$ be a differentiable map and introduce $H_t:= A+U_tBU_t^*$. Then by the above we must have that $\E [M(H_t)]$ is constant in $t$ and hence 
\begin{align}\label{le how to get ward}
 \frac{{\rm d}}{{\rm d} t} \bigg|_{t=0}\E [M(H_t)]=0\,.
\end{align}
 
 In view of~\eqref{le semi expression}, we will first choose 
 \begin{equation}\label{function_invariant}
M(H_t)=e_0(\lambda,t) \big( G_t(z_1)\big)_{ij}=\exp \Big\{ \frac{\i \lambda}{\pi} \int_{\Omega_2}\frac{\partial }{\partial \overline{z_2}} \tf(z_2) \Big( \Tr G_t(z_2)-\E \Tr G_t(z_2)\Big) \dd^2 z_2  \Big\} \big(G_t(z_1)\big)_{ij},
\end{equation} 
where $G_t:=(H_t-zI)^{-1}$ and $e_0(\lambda,t)$ is given by \eqref{e22} with $G$ replaced by $G_t$. Using
\begin{equation}\label{derivative}
\frac{{\rm d}}{{\rm d} t}\bigg|_{t=0} H_t= \frac{{\rm d}}{{\rm d} t}\bigg|_{t=0} \ U_t B U_t^*=\i [X, \tb],\qquad \frac{{\rm d}}{{\rm d} t}\bigg|_{t=0} G_t=-\i G [X, \tb] G,
\end{equation}
we obtain from~\eqref{le how to get ward} with the choice of $M$ in~\eqref{function_invariant} the relation
$$\E\Big[ \Big( \ea \frac{\i \lambda}{\pi} \int_{\Omega_2} \frac{\partial }{\partial \overline{z_2}} \tf(z_2) \sum_{l=1}^N \Big( -\i G(z_2) [X, \tb] G(z_2)\Big)_{ll} \dd^2 z_2 \Big) (G(z_1))_{ij}\Big] =\E \Big[\ea \Big( \i  G(z_1) [X, \tb] G(z_1)  \Big)_{ij}\Big],$$
where $X$ is an arbitrary deterministic self-adjoint matrix.

Let now first $X=\ei \ej^*+\ej \ei^*$ and then $X=\i \ei \ej^* -\i \ej \ei^*$. Using linearity and averaging over the index $i$, we obtain, for fixed $j$ and $z_1\in D_{bulk}$, the following identity
\begin{equation}\label{equation_2}
\E \Big[  \ea \Big(\underline{\tb G} G_{jj}-\g (\tb G)_{jj} \Big) \Big]=I_j(z_1)\,,
\end{equation}
with 
\begin{align}\label{le Ij}
 I_j(z_1):=&\frac{1}{N}\E \Big[ \ea \frac{\i \lambda}{\pi} \int_{\Omega_2} \frac{\partial }{\partial \overline{z_2}} \tf(z_2) \frac{\partial }{\partial {z_2}} \Big( (\tb F G)_{jj} -(F \tb G)_{jj}\Big) \dd^2 z_2  \Big]\,
\end{align}
where we further introduced the short hands
\begin{align}\label{le shorthands}
F \equiv G(z_2)\,,\qquad G \equiv G(z_1)\,.
\end{align}

We first work on the left side of~\eqref{equation_2}. Repeating the above invariance argument with $M(H_t)=(G_t(z_1))_{ij}$, we obtain after averaging over the index $i$ the identity
\begin{align}\label{le another ward}
 \E \Big(\underline{\tb G} G_{jj}-\g (\tb G)_{jj} \Big)=0\,,
\end{align}
and we can write
\begin{align}\label{le table}
 \E \Big[  \ea \Big(\underline{\tb G} G_{jj}-\g (\tb G)_{jj} \Big) \Big]&=\E \Big[  \ea\Big( \<\underline{\tb G} G_{jj}\>-\<\g (\tb G)_{jj} \>\Big) \Big]\,.
\end{align}

Next, recalling the definition of the resolvent in~\eqref{tilde quantities}, we write $(a_j-z_1)G_{jj}+(\wt{B}G)_{jj}=1$, which implies
\begin{equation}\label{resolvent_def}
\E [\ea \<\g(\tb G)_{jj}\>]=(z_1-a_j) \E [\ea \<\g G_{jj}\>]+\E [\ea \<\g\>]\,.
\end{equation}
We then rewrite~\eqref{le table} as
\begin{align}\label{le table 2}
 \E \Big[  \ea \Big(\underline{\tb G} G_{jj}-\g (\tb G)_{jj} \Big) \Big]&=\E \Big[  \ea\Big( \<\underline{\tb G} G_{jj}\>+(a_j-z_1)\<\g  G_{jj} \>-\<\g\>\Big) \Big]\,.
\end{align}
Next, in view of the local laws Theorem~\ref{local}, we write the right side of~\eqref{le table 2} as
\begin{align}
\E \Big[  \ea\Big(& \<\underline{\tb G} G_{jj}\>+(a_j-z_1)\<\g  G_{jj} \>-\<\g\>\Big) \Big]\nonumber\\
&=(z_1-\omega_B(z_1))m_{fc}(z_1)\E \big[  \ea \<G_{jj}\>\big]+(a_j-z_1)m_{fc}(z_1)\E\big[\ea\<  G_{jj} \> \big]\nonumber\\
&\qquad+\E \Big[  \ea\Big( \<\big(\underline{\tb G}-(z_1-\omega_B(z_1))m_{fc}(z_1)\big) G_{jj}\> \Big]\nonumber\\
&\qquad+(a_j-z_1)\E \Big[  \ea\<\big(\g-m_{fc}(z_1))  G_{jj} \>\Big) \Big]-\E [\ea \<\g\>]\nonumber\\
&=(a_j-\omega_B(z_1))m_{fc}(z_1)\E \big[  \ea \<G_{jj}\>\big]-\E [\ea \<\g\>]\nonumber\\
&\qquad+\E \Big[  \ea \<\big(\underline{\tb G}-(z_1-\omega_B(z_1))m_{fc}(z_1)\big) G_{jj}\> \Big]\nonumber\\
&\qquad+(a_j-z_1)\E \Big[  \ea\<\big(\g-m_{fc}(z_1))  G_{jj} \>\Big]\,.
\end{align}
Returning to~\eqref{equation_2}, we hence obtain, after rearranging,
\begin{align}
 &(a_j-\omega_B(z_1))m_{fc}(z_1)\E \big[  \ea \<G_{jj}\>\big]-\E [\ea \<\g\>] \nonumber\\
 &\qquad\qquad=I_j(z_1)-\E \Big[  \ea \<\big(\underline{\tb G}-(z_1-\omega_B(z_1))m_{fc}(z_1)\big) G_{jj}\> \Big]\nonumber\\
&\qquad\qquad\qquad\qquad-(a_j-z_1)\E \Big[  \ea\<\big(\g-m_{fc}(z_1))  G_{jj} \> \Big]\,.
\end{align}
Dividing by~$(a_j-\omega_B(z_1))$ and then summing over the index $j$, the left side of the above equation vanishes by (\ref{self_m_fc}), and we thus obtain
\begin{align}
\sum_{j=1}^N\frac{I_j(z_1)}{a_j-\omega_B(z_1)}&=\E \Big[  \ea \<\big(\underline{\tb G}-(z_1-\omega_B(z_1))m_{fc}(z_1)\big)\sum_{j=1}^N\frac{1}{a_j-\omega_B(z_1)} G_{jj}\> \Big]\nonumber\\
&\qquad+\E \Big[  \ea\<\big(\g-m_{fc}(z_1))\sum_{j=1}^N\frac{a_j-z_1}{a_j-\omega_B(z_1)}  G_{jj} \> \Big]\,.
\end{align}
At this point, we invoke the local laws in~Theorem~\ref{local}, to get the estimate
\begin{align}
\sum_{j=1}^N\frac{I_j(z_1)}{a_j-\omega_B(z_1)}&=\E \Big[  \ea \<\big(\underline{\tb G}-(z_1-\omega_B(z_1))m_{fc}(z_1)\big)\sum_{j=1}^N\frac{1}{(a_j-\omega_B(z_1))^2} \> \Big]\nonumber\\
&\qquad+\E \Big[  \ea\<\big(\g-m_{fc}(z_1))\sum_{j=1}^N\frac{a_j-z_1}{(a_j-\omega_B(z_1))^2}  \>\Big]+O_\prec\big(\frac{1}{N\eta_1^2}\big)\nonumber\\
&=\frac{1}{N}\sum_{j=1}^N\frac{1}{(a_j-\omega_B(z_1))^2}\E \Big[  \ea \< \Tr({\tb G}) \> \Big]\nonumber\\
&\qquad+\frac{1}{N}\sum_{j=1}^N\frac{a_j-z_1}{(a_j-\omega_B(z_1))^2} \E \Big[  \ea\<\Tr G \> \Big]+O_\prec\big(\frac{1}{N\eta_1^2}\big)\,,
\end{align}
where the second equality follows from the property that $\<\mathcal{X}\>=0$ if $\mathcal{X}$ is deterministic. Next, note that
\begin{align}
 \frac{1}{N}\sum_{j=1}^N\frac{a_j-z_1}{(a_j-\omega_B(z_1))^2}&=\frac{1}{N}\sum_{j=1}^N\frac{a_j-\omega_B(z_1)}{(a_j-\omega_B(z_1))^2}+\frac{1}{N}\sum_{j=1}^N\frac{\omega_B(z_1)-z_1}{(a_j-\omega_B(z_1))^2}\nonumber\\
 &=m_{fc}(z_1)+m_A'(\omega_B(z_1))(\omega_B(z_1)-z_1)\,.
\end{align}
Hence we obtain
\begin{align}\label{sum11}
 &\big( m_{fc}(z_1)-(z_1-\omega_B(z_1))m_A'(\omega_B(z_1)) \big)\E \big[  \ea \<\Tr G\>\big]\nonumber\\
 &=\sum_{j=1}^N\frac{I_j(z_1)}{a_j-\omega_B(z_1)}-m_A'(\omega_B(z_1))\E \Big[  \ea \< \Tr({\tb G}) \> \Big]+O_\prec\big(\frac{1}{N\eta_1^2}\big)\,.
\end{align}

Next, we treat $\tg$ in (\ref{tilde quantities}) similarly and obtain
\begin{align}\label{sum22}
 &\big( m_{fc}(z_1)-(z_1-\omega_A(z_1))m_B'(\omega_A(z_1)) \big)\E \big[  \ea \<\Tr \mathcal{G}\>\big]\nonumber\\
 &=\sum_{j=1}^N\frac{\mathcal{I}_j(z_1)}{b_j-\omega_A(z_1)}-m_B'(\omega_A(z_1))\E \Big[  \ea \< \Tr({\ta \mathcal{G}}) \> \Big]+O_\prec\big(\frac{1}{N\eta_1^2}\big)\,,
\end{align}
where we wrote $\mathcal{F} \equiv \mathcal{G}(z_2)$, $\tg \equiv \tg(z_1)$ for short and introduced
\begin{align}\label{le Ij tilde}
 \mathcal{ I}_j(z_1):= \frac{1}{N}\E \Big[ \ea \frac{\ii \lambda}{ \pi} \int_{\Omega_2}  \frac{\partial }{\partial \overline{z_2}} \tf(z_2) \frac{\partial }{\partial {z_2}} \Big( (\ta \mathcal{F} \tg)_{jj} -(\mathcal{F} \ta \tg)_{jj} \Big) \dd^2 z_2  \Big].
\end{align}

Combining (\ref{sum11}) and (\ref{sum22}) with the definition of the resolvent~\eqref{tilde quantities} and the subordination equations (\ref{self_m_fc}), we obtain
\begin{align}
 m^3_{fc}(z_1) \Delta(z_1) \E [\ea \<\Tr G\>]  =m'_B(\omega_A(z_1)) \sum_{j=1}^N\frac{I_j(z_1)}{a_j-\omega_B(z_1)}  +m'_A(\omega_B(z_1)) \sum_{j=1}^N\frac{\mathcal{I}_j(z_1)}{b_j-\omega_A(z_1)} +O_{\prec} \Big( \frac{1}{N \eta_1^2} \Big),\label{sum33}
\end{align}
with $\Delta(z)$ given in~(\ref{delta}). 
Recall that we have $z_1\in D_{bulk}$, hence by (\ref{imaginary_bound}), (\ref{delta_bound}) and (\ref{differential_relation}), we can divide (\ref{sum33}) by $m^3_{fc}(z_1) \Delta(z_1)$ to obtain
\begin{align}\label{sum1}
\E [\ea \<\Tr G\>]=&\frac{\omega_B'(z_1)}{m_{fc}(z_1)}\sum_{j=1}^N\frac{I_j(z_1)}{a_j-\omega_B(z_1)}  +\frac{\omega_A'(z_1)}{m_{fc}(z_1)} \sum_{j=1}^N\frac{\mathcal{I}_j(z_1)}{b_j-\omega_A(z_1)}  +O_{\prec}\Big( \frac{1}{N \eta_1^2  }\Big)\nonumber\\
=&\frac{ \ii \lambda}{ \pi} \E \bigg[ \ea  \int_{ \Omega_2} \frac{\partial }{\partial \overline{z_2}} \tf(z_2) \frac{\partial }{\partial {z_2}}K(z_1,z_2) \dd^2 z_2 \bigg]+O_{\prec}\Big( \frac{1}{N \eta_1^2 }\Big),
\end{align}
where $K(z_1,z_2)$ is given by
\begin{equation}\label{kernel_K_1}
K(z_1,z_2):= \frac{\omega_B'(z_1)}{m_{fc}(z_1)} (K_{B,1}-K_{B,2})+\frac{\omega_A'(z_1)}{m_{fc}(z_1)} (K_{A,1}-K_{A,2}),
\end{equation}
with
$$K_{B,1}:=\frac{1}{N} \sum_{j=1}^N \frac{1}{a_j-\omega_B(z_1)}  (\tb F G)_{jj}; \quad  K_{B,2}:=\frac{1}{N} \sum_{j=1}^N \frac{1}{a_j-\omega_B(z_1)} (F \tb G)_{jj};$$
$$K_{A,1}:=\frac{1}{N} \sum_{j=1}^N \frac{1}{b_j-\omega_A(z_1)}  (\ta \mathcal{F} \mathcal{G})_{jj}; \quad  K_{A,2}:=\frac{1}{N} \sum_{j=1}^N \frac{1}{b_j-\omega_A(z_1)} (\mathcal F \ta \mathcal G)_{jj}.$$
This completes the proof of Lemma~\ref{lemma_useful}.
\end{proof}

\subsection{Proof of Proposition~\ref{prop}}

 The two terms $K_{A,1}$ and $K_{B,1}$ are easy to identify: Using the resolvent identity~\eqref{le rock 4}, we have
\begin{align*}
K_{B,1}&=\frac{1}{z_1-z_2}\frac{1}{N}\sum_{j=1}^N \frac{1}{a_j-\omega_B(z_1)} \Big( (\tb G)_{jj}(z_1)-(\tb G)_{jj}(z_2) \Big).
\end{align*}
Recall the local law (\ref{average}) and choose $d_j(z)=\frac{1}{a_j-\omega_B(z)}$. The uniform bound of $(d_j(z))_j$ follows from (\ref{imaginary_bound}). Though $(d_j(z))_j$ depends on $z$,  we can use a continuity argument to show that the local law still holds, i.e.,
\begin{equation}\label{argument_continuity}
\Big| \frac{1}{N}\sum_{j=1}^N \frac{1}{a_j-\omega_B(z_1)} \Big( (\tb G(z_l))_{jj} -\frac{z_l-\omega_B(z_l)}{a_j-\omega_B(z_l)} \Big) \Big| \prec \frac{1}{N |\eta_l|}, \qquad l=1,2.
\end{equation}
For notational simplicity, we define
$$h_j(z):=(\tb G(z))_{jj} -\frac{z-\omega_B(z)}{a_j-\omega_B(z)}.$$
Next, we consider two cases to show that
\begin{equation}\label{argument_1}
\Big|\frac{1}{z_1-z_2}\frac{1}{N}\sum_{j=1}^N \Big( d_j(z_1) (h_j(z_1)-h_j(z_2)) \Big)\Big| \prec \frac{1}{N \eta^2_1}+\frac{1}{N |\eta_1 \eta_2|}.
\end{equation}

	\noindent \textbf{Case 1:} If $z_1$ and $z_2$ are in different half-planes, then $\frac{1}{|z_1-z_2|} \leq \frac{1}{| {\Im}z_1 |}.$ Thus (\ref{argument_1}) follows directly from (\ref{argument_continuity}).
	
	\noindent \textbf{Case 2:} If $z_1$ and $z_2$ are in the same half-plane, without loss of generality, we can assume they both belong to the upper half plane. If $|{\Im}z_1-{\Im}z_2| \geq \frac{1}{2} {\Im}z_1$, then we can use the same argument as in Case~1. Thus it is sufficient to consider $|{\Im}z_1-{\Im}z_2| \leq \frac{1}{2}{\Im}z_1$, which means $ \frac{2}{3} {\Im}z_2 \leq {\Im}z_1 \leq 2 {\Im}z_2$. The left side of (\ref{argument_1}) can be bounded as
	\begin{align*}
	\Big|  \frac{1}{N}\sum_{j=1}^N  \frac{d_j(z_1)}{z_1 -z_2}  (h_j(z_1)-h_j(z_2)) \Big| & \leq \Big| \frac{1}{N}\sum_{j=1}^N \frac{d_j(z_1)-d_j(z_2)}{z_1-z_2} h_j(z_2)  \Big|\nonumber\\ &\qquad\qquad +\Big| \frac{\frac{1}{N}\sum_{j=1}^N d_j(z_1) h_j(z_1)-\frac{1}{N}\sum_{j=1}^N d_i(z_2) h_j(z_2)) }{z_1-z_2} \Big|.
	\end{align*}
	The coefficients of the first term on the right side have the following upper bound
	$$\Big| \frac{d_j(z_1)-d_j(z_2)}{z_1-z_2} \Big| = |d_j(z_1) || d_{j}(z_2) | \Big| \frac{\omega_{B}(z_1)-\omega_{B}(z_2)}{z_1-z_2} \Big|\leq C.$$
	The last step follows from the fact that $\omega_B(z)$ is analytic in the neighborhood of the segment connecting $z_1$ and $z_2$ and (\ref{omega_prime_bound}). Using the arguments in proving (\ref{argument_continuity}), one shows from the local law (\ref{average}) that the first term is bounded by $O_{\prec}\Big(\frac{1}{N |\eta_2|}\Big).$ The second term is bounded by $O_{\prec}\Big(\frac{1}{N \eta^2_1}+\frac{1}{N |\eta_1\eta_2|}\Big)$ from (\ref{argument_continuity}) using the Cauchy integral formula. Thus the error term has the same upper bound as in Case 1.

	Therefore, by direct computation, we obtain that
\begin{align}
K_{B,1}=&\frac{(z_1-\omega_B(z_1)) L_B(z_1,z_1) -(z_2-\omega_B(z_2)) L_B(z_1,z_2)}{z_1-z_2}+O_{\prec}\Big( \frac{1}{N \eta^2_1}\Big)+O_{\prec}\Big( \frac{1}{N |\eta_1 \eta_2|}\Big), \label{right_first_a}
\end{align}
and similarly,
\begin{align}
K_{A,1}=&\frac{(z_1-\omega_A(z_1)) L_A(z_1,z_1) -(z_2-\omega_A(z_2)) L_A(z_1,z_2)}{z_1-z_2}+O_{\prec}\Big( \frac{1}{N \eta^2_1}\Big)+O_{\prec}\Big( \frac{1}{N |\eta_1 \eta_2|}\Big),\label{right_first_b}
\end{align}
with $L_A(z_1,z_2)$ and $L_B(z_1,z_2)$ given in (\ref{LM_def}).

Next, we estimate the rest two terms $K_{B,2}$ and $K_{A,2}$. Note that the Gromov-Milman concentration inequality (see, e.g., \cite{concentrate, stability}) is not sufficiently strong to obtain the optimal mesoscopic CLT in the regular bulk. We will use the random partial decomposition used in \cite{eta1} to prove the following lemma in the next subsection.

\begin{proposition}\label{variance_bfbg}
Under the same assumptions as in  Proposition \ref{prop}, there exists a small neighborhood of $E_0$, denoted by $D_0$, such that for all  $z_{1}=E_1+\ii \eta_1, z_2=E_2+\ii \eta_2 \in D_0 \cap D_{bulk}$, we have the following estimate for every $j$:
\begin{equation}\label{fbg_ii}
 (G(z_2) \tb G(z_1))_{jj} = \frac{1}{(z_1-z_2)  (a_j-\omega_B(z_1))(a_j-\omega_B(z_2))} \frac{T_B(z_1,z_2)}{L_B(z_1,z_2)}+E_{B,j}(z_1,z_2),
\end{equation}
for sufficiently large $N$, where $L_{B}(z_1,z_2)$ is given in (\ref{LM_def}), and 
\begin{equation}\label{T_B}
T_{B}(z_1,z_2):=(z_1-\omega_{B}(z_1))m_{fc}(z_1)- (z_2-\omega_{B}(z_2))m_{fc}(z_2),
\end{equation}
and the error function $E_{B,j}(z_1,z_2)$ is analytic in $z_1, z_2 \in \C \setminus \R$ with the following estimate:
\begin{equation}\label{bound_notation}
E_{B,j}(z_1,z_2)= O_{\prec}\Big(\frac{1}{\sqrt{N |\eta_1|} |\eta_2|} +\frac{1}{\sqrt{N |\eta_2|} |\eta_1|}+\frac{1}{N |\eta_1\eta_2|}+\frac{1}{N \eta_1^2}\Big).
\end{equation}
In addition, if $z_1$ and $z_2$ belong to different half planes, that is, for all $z_1, \overline{z_2} \in D_{bulk}$, (\ref{fbg_ii}) also holds true. 

The same holds true for $(\mathcal{G}(z_2) \ta \mathcal{G}(z_1))_{jj}$ with $\tg$ in (\ref{tilde quantities}) by interchanging the roles of $A$ and $B$.
\end{proposition}
Recalling (\ref{sum1bis}) and applying Stokes' formula, we obtain
\begin{align*}
\E [\ea \<\Tr G(z_1)\>]=&\frac{ \lambda}{2 \pi} \E \Big[ \ea  \int_{ \partial \Omega_2}  \tf(z_2) \frac{\partial}{\partial z_2} K(z_1,z_2) \dd z_2 \Big]+O_{\prec}\Big( \frac{1}{N \eta_1^2 }\Big)\,.
\end{align*}
We further apply the above equation to (\ref{newphi}), using Stokes' formula and (\ref{assumpf}), we have
\begin{equation}\label{plug_to}
\phi'(\lambda)=\frac{\lambda}{4\pi^2}\E \Big[ \ea \int_{\partial \Omega_1}  \tf(z_1)  \int_{ \partial \Omega_2}  \tf(z_2) \frac{\partial}{\partial z_2} K(z_1,z_2) \dd z_2 \dd z_1\Big]+O_{\prec}\Big( \frac{N^{2\tau}}{N \eta_0 }\Big)+O_{\prec}\Big( |\lambda| (\log N) N^{ -\tau} \Big).
\end{equation}
If $z_1 \in \partial \Omega_1$ and $z_2 \in \partial \Omega_2$ and they are in the same half plane, then by (\ref{fn}), $z_1,z_2 \in D_0 \cap D_{bulk}$ or $\overline{z_1}, \overline{z_2} \in D_0 \cap D_{bulk}$ (see Proposition \ref{variance_bfbg}) for $N$ sufficiently large. Combining with the fact that $|a_j-\omega_B(z)| \geq \Im \omega_B(z)>c>0$ by (\ref{imaginary_bound}), we hence obtain 
\begin{equation}\label{right_second_a}
K_{B,2}=\frac{1}{z_1-z_2} \frac{1}{N} \sum_{j=1}^N \frac{1}{ (a_j-\omega_B(z_1))^2(a_j-\omega_B(z_2))} \frac{T_B(z_1,z_2)}{L_B(z_1,z_2)}+E_{B}(z_1,z_2),
\end{equation}
and similarly, 
\begin{equation}\label{right_second_b}
K_{A,2}=\frac{1}{z_1-z_2} \frac{1}{N} \sum_{j=1}^N \frac{1}{ (b_j-\omega_A(z_1))^2(b_j-\omega_A(z_2))} \frac{T_A(z_1,z_2)}{L_A(z_1,z_2)}+E_{A}(z_1,z_2),
\end{equation}
where the error functions $E_{A}(z_1,z_2), E_{B}(z_1,z_2)$ are analytic in $z_1, z_2 \in \C \setminus \R$ with the same upper bound as in (\ref{bound_notation}).
Note that by (\ref{LM_def}) and (\ref{LM}), we have
\begin{align*}
\frac{1}{N}\sum_{j=1}^N \frac{1}{(a_j-\omega_B(z_1))^2(a_j-\omega_B(z_2))}&=\frac{1}{\omega'_B(z_1)}\frac{\partial}{\partial z_1} L_B(z_1,z_2)\nonumber\\ &= \frac{m'_{fc}(z_1)}{\omega'_B(z_1)} \frac{1}{\omega_B(z_1)-\omega_B(z_2)}-\frac{m_{fc}(z_1)-m_{fc}(z_2)}{(\omega_B(z_1)-\omega_B(z_2))^2}.
\end{align*}
Plugging (\ref{right_first_a}), (\ref{right_first_b}), (\ref{right_second_a}) and (\ref{right_second_b}) into (\ref{kernel_K_1}), together with (\ref{LM}), (\ref{T_B}) and (\ref{LMprime}), we have
\begin{equation*}
K(z_1,z_2)= \frac{\omega'_A(z_1)}{\omega_A(z_1)-\omega_A(z_2)}+\frac{\omega'_B(z_1)}{\omega_B(z_1)-\omega_B(z_2)}-\frac{1}{z_1-z_2}-\frac{m_{fc}'(z_1) m_{fc}(z_2)}{m_{fc}(z_1)(m_{fc}(z_1)-m_{fc}(z_2))}+E(z_1,z_2),
\end{equation*}
where the error term $E(z_1,z_2)$ has an upper bound from (\ref{m_bound}), (\ref{imaginary_bound}), (\ref{omega_prime_bound}) and (\ref{bound_notation}):
\begin{equation}
|E(z_1,z_2)|=O_{\prec}\Big( \frac{1}{\sqrt{N |\eta_1|} |\eta_2|} +\frac{1}{\sqrt{N |\eta_2|} |\eta_1|}+\frac{1}{N \eta^2_1}+\frac{1}{N |\eta_1\eta_2|} \Big).
\end{equation}
Since $E(z_1,z_2)$ is analytic in $z_1, z_2 \in \C \setminus \R$, the Cauchy integral formula yields
\begin{align}\label{plug}
\frac{\partial }{\partial {z_2}}K(z_1,z_2)=\mathcal{K}(z_1,z_2)+\mathcal{E}(z_1,z_2),
\end{align}
with the kernel function
\begin{align}
\mathcal K(z_1,z_2)&:= \frac{\omega'_A(z_1)\omega'_A(z_2)}{(\omega_A(z_1)-\omega_A(z_2))^2}+\frac{\omega'_B(z_1)\omega'_B(z_2)}{(\omega_B(z_1)-\omega_B(z_2))^2}-\frac{1}{(z_1-z_2)^2}-\frac{m_{fc}'(z_1) m'_{fc}(z_2)}{(m_{fc}(z_1)-m_{fc}(z_2))^2}\nonumber\\
&=\pzab \log \Big( \frac{(\omega_A(z_1)-\omega_A(z_2)) (\omega_B(z_1)-\omega_B(z_2)) }{(z_1-z_2)(\frac{1}{m_{fc}(z_2)}-\frac{1}{m_{fc}(z_1)})} \Big)=- \pzab \log (\Delta(z_1,z_2)),
\end{align}
with $\Delta(z_1,z_2)$ in (\ref{delta2}), and the error function
\begin{equation}\label{error_bound}
\mathcal{E}(z_1,z_2)=O_{\prec} \Big(\frac{1}{\sqrt{N| \eta_1|} \eta^2_2} +\frac{1}{\sqrt{N |\eta_2|} |\eta_1 \eta_2|}+\frac{1}{N \eta^2_1|\eta_2|}+\frac{1}{N |\eta_1|\eta^2_2}\Big).
\end{equation}
Plugging (\ref{plug}) into (\ref{plug_to}), by (\ref{error_bound}) and (\ref{assumpf}), we have
$$\phi'(\lambda)=\lambda \E [ \ea] \frac{1}{4 \pi^2} \int_{\partial \Omega_1}\int_{\partial \Omega_2}   \tf(z_1)  \tf(z_2)  \mathcal{K}(z_1,z_2) \dd z_2   \dd z_1+O_{\prec}\Big(\frac{N^{3 \tau}}{\sqrt{N \eta_0}} \Big)+O_{\prec}( |\lambda|  N^{ -\tau}).$$
The last step is to replace $\ea$ by $e(\lambda)$ with difference (\ref{e2}), provided that $V(f) \prec 1$. Thus we finish the proof of Proposition \ref{prop}.

\subsection{Proof of Theorem \ref{meso}}
We end this section by computing the explicit formula of $V(f)$ in (\ref{vf}), where the test function $f$ is given in (\ref{fn}). If $z_1 \in \Gamma_1$ and $z_2 \in \Gamma_1$ are in the same half plane, using the expansions of $\omega_A(z)$, $\omega_B(z)$ and $m_{fc}(z)$ near $E_0$ inside the regular bulk such that $|m_{fc}'(z)|, |\omega_A'(z)|, |\omega_{B}'(z)| \sim 1$, we have 
$$\mathcal K(z_1,z_2)=O(\frac{1}{|z_1-z_2|}) =O(\eta_0^{-1} N^{\tau}).$$ 
Combining with (\ref{assumpf}), the integral for $z_1$ and $z_2$ belonging to the same half plane only contributes $O(\eta_0N^{\tau})$. If $z_1$ and $z_2$ are in different half planes, since $E_0$ is in the regular bulk, we have $|m_{fc}(z_1)-m_{fc}(z_2)| \geq c>0$, as well as $|\omega_A(z_1)-\omega_A(z_2)|$, $|\omega_B(z_1)-\omega_B(z_2)|$. Hence
$$\mathcal K(z_1,z_2) =-\frac{1}{(z_1-z_2)^2}+O(1).$$ 
Since the computation in the following is similar as the proof of Lemma 6.1 in \cite{Li+Schnelli+Xu}, we omit it here. Therefore, we have
$$\lim_{N \rightarrow \infty}V(f)=\frac{1}{4 \pi^2} \int_{\R} \int_{\R} \frac{(g(x)-g(y))^2}{(x-y)^2} \dd x \dd y.$$
Thus $V(f)$ converges to some positive constant since $g \in C^2_c(\R)$. Theorem \ref{meso} is a direct result of Proposition \ref{prop} after integrating $\phi'(\lambda)$ and using the L\'{e}vy continuity theorem.
 Hence we finish the proof of Theorem \ref{meso}.

\section{Proof of Proposition \ref{variance_bfbg}}\label{sec_proof_main_lemma}
In this section, we use a partial randomness decomposition to prove Proposition \ref{variance_bfbg}. We remark that this decomposition was a key ingredient in~\cite{eta1,eta2} to derive the local laws in Theorem~\ref{local}. For any Haar unitary matrix $U \equiv U_N \in U(N)$,  there exists a random vector $\vi$, the $i$-th column of the matrix $U$ and  an independent Haar unitary matrix $U^{i}\in U(N-1)$, such that
$$U=-e^{\i \theta_i} R_i U^{\<i\>};\qquad R_i:=I-\ri \ri^*;\qquad \ri:=\sqrt{2} \frac{\ei+e^{-\ii \theta_i} \vi}{\|\ei+e^{-\ii \theta_i} \vi\|_2},$$
where $\theta_i$ is the argument of $i$-th entry of $\vi$, denoted by $v_{ii}$, $R_i$ is the Householder transform sending $\ei$ to $-e^{-\i \theta_i}{\bm v_i}$, and $U^{\<i\>}$ is a unitary matrix with $\ei$ as its $i$-th column and $U^{i}$ as its $(i,i)$-minor. Thus we can write
$$\tb=U B U^* =R_i \tb^{\<i\>} R_i, \quad \mbox{with} \quad \tb^{\<i\>}:=U^{\<i\>} B (U^{\<i\>})^*.$$
Note that $\tbhat$ is independent of $\vi$, and define
\begin{equation}\label{hat_definition}
H^{\<i\>}:=A+\tbhat; \qquad  \ghat(z):=(A+\tbhat-z)^{-1}.
\end{equation}
It is well-known that $\vi$ is a uniformly distributed unit vector in $\C^{N}$, and there exists a Gaussian vector $\widetilde \gi \sim {N}_{\C}(0,N^{-1}I_N)$, such that 
\begin{equation}\label{uniform_vector}
\vi=\frac{\widetilde \gi}{\|\widetilde \gi\|_2}.
\end{equation}
 Hence we can write 
$$\hi:=e^{-\ii \theta_i}\vi=e^{-\ii \theta_i} \frac{\widetilde{\gi}}{\|\widetilde \gi\|_2};  \qquad \ri:=l_i (\ei+{\bm h_i}), \qquad \mbox{with } l_i :=\frac{\sqrt{2}}{\|\ei+{\bm h_i}\|_2}=1+O_{\prec}(\frac{1}{\sqrt{N}}).$$
Note that $\hi$ is independent of $ \tb^{\<i\>}$, and 
$$R_i \ei=-{\bm h_i}; \qquad R_i {\bm h_i}=-\ei; \qquad {\bm h_i}^* \tb^{\<i\>} R_i=-\ei^* \tb; \qquad \ei^* \tb^{\<i\>} R_i=-b_i {\bm h_i}^*=-{\bm h_i}^* \tb .$$

Set $g_{ik}:=e^{-\ii \theta_i} \widetilde{g}_{ik}$ for $k \neq i$
where $g_{ik} \sim {N}_{\C}(0,N^{-1})$ are independent Gaussian random variables, and we further introduce an independent Gaussian random variable $g_{ii} \sim {N}_{\C}(0,N^{-1})$. To simplify the proof, we use the Gaussian vector $\gi:=(g_{i1}, \cdots g_{iN})\sim {N}_{\C}\Big(0,\frac{1}{N}I_N\Big)$ to approximate $\hi$, and define
$$ \tb^{(i)}=W_i \tb^{\<i\>} W_i; \qquad \mbox{with  }W_i:=I-{\bm w_i} {\bm w_i}^*, \qquad {\bm w_i}:=\ei+\gi,$$
and 
\begin{equation}\label{bra_definition}
H^{(i)}:=A+B^{(i)}; \qquad G^{(i)}(z):=(A+B^{(i)}-z)^{-1}.
\end{equation}
Note that we have 
\begin{equation}\label{h_approximate_r}
\hi=\frac{|\widetilde g_{ii}|-g_{ii}}{\|\widetilde \gi\|_2} \ei+ \frac{1}{\|\widetilde \gi\|_2} \gi; \qquad \ri={\bm w_i}+d_{1} \ei +d_2 \gi,
\end{equation} 
where 
\begin{equation}\label{bound_d}
d_1:=( l_i -1)+l_i \frac{|\widetilde g_{ii}|-g_{ii}}{\|\widetilde \gi\|_2}=O_{\prec} \Big(\frac{1}{\sqrt{N}}\Big); \qquad  d_2 := \frac{l_i}{\|\widetilde \gi\|_2} -1=O_{\prec}\Big( \frac{1}{\sqrt{N}}\Big).
\end{equation}
Because of this, $G^{(i)}(z)$ is a good approximation of $G(z)$, see (\ref{removei_1}) below. 
\begin{lemma}[Lemma 4.1 in \cite{eta1}]
For any $1 \leq i,j, k\leq N$ and $z=E+\ii \eta \in D_{bulk}$, we have 
\begin{equation}\label{removei_1}
\Big| G_{jk}(z)-G^{(i)}_{jk}(z)\Big| , \qquad \Big| (\tb G)_{jk}-(\tbbra G^{(i)})_{jk}\Big| \prec \frac{1}{\sqrt{N \eta}}.
\end{equation}
\end{lemma}

Before we proceed with the proof, we first introduce some previous results that will be used later.
\begin{lemma}[Corollary 5.2, Propositions 6.1 and 8.3 in \cite{eta1}]\label{previous_bound}
For $j \neq i$, define
\begin{equation}\label{define_S_one}
S_i^{[1]}:=\gi^* \tb^{\<i\>}G^{(i)} \ei; \qquad S_{i,j}^{[1]}:=\gi^* \tb^{\<i\>}G^{(i)} \ej; \qquad T_i^{[1]}:=\gi^*G^{(i)} \ei; \qquad T_{i,j}^{[1]}:=\gi^* G^{(i)} \ej.
\end{equation}
Then for all $z=E+\ii \eta \in D_{bulk}$, we have the following estimates
$$S_{i}^{[1]}=-\frac{z-\omega_B(z)}{a_i-\omega_B(z)}+O_{\prec}\Big(\frac{1}{\sqrt{N \eta}}\Big); \qquad T_i^{[1]},\quad T_{i,j}^{[1]}, \quad S_{i,j}^{[1]}, \quad \gbra_{ij} =O_{\prec}\Big(\frac{1}{\sqrt{N \eta}}\Big).$$
Hence we obtain the following local laws:
\begin{equation}\label{local_law_lemma}
(\tbbra \gbra)_{ii}=-S_{i}^{[1]}+O_{\prec}\Big(\frac{1}{\sqrt{N}}\Big)=\frac{z-\omega_B}{a_i-\omega_B}+O_{\prec}\Big(\frac{1}{\sqrt{N\eta}}\Big); \qquad \gbra_{ii}=\frac{1}{a_i-\omega_B}+O_{\prec}\Big(\frac{1}{\sqrt{N\eta}}\Big).
\end{equation}
In addition, for $\xii, \yi$ either $\gi$ or $\ei$, and $Q^{\<i\>}_1,Q_2^{\<i\>}$ either $\tbhat$ or $I$, we have an upper bound:
$$\max_{i} |\xii^*  Q^{\<i\>}_1 \gbra Q^{\<i\>}_2\yi | \prec 1.$$
Moreover, for all $1 \leq i,j \leq N$, we have
\begin{equation}\label{trivial}
|\gi^* \ej|, \qquad |\gi^* \tbhat  \ej|, \qquad |\gi^* \tbhat  \gi| \prec \frac{1}{\sqrt{N}}.
\end{equation}
\end{lemma}

In the proof of the local laws in (\ref{local_law_lemma}), the following two lemmas were introduced in \cite{eta1}. Recall the shorthand notation~\eqref{notation_X}.
\begin{lemma}[Lemma 5.1, Corollary 5.2 in \cite{eta1}]\label{lemma_concentrate}
For all $z=E+\ii \eta \in D_{bulk}$ and $1 \leq i \leq N$, we have
\begin{equation}\label{concentrate_1}
|\IE S_i^{[1]}|,\quad |\IE  T_i^{[1]}|,\quad |\IE  G_{ii}^{(i)}| \prec \frac{1}{\sqrt{N \eta}},
\end{equation}
where we use the notation that for any general random variable~$\mathcal{X}$,
$$\IE \mathcal{X}:=\mathcal{X}-\E_{\gi} \mathcal{X},$$
denoting by $\E_{\gi}$ the partial expectation with respect to the Gaussian vector $\bm{g_i}$.
\end{lemma} 

\begin{lemma}[Lemma 6.2 in \cite{eta1}]
For all $z=E+\ii \eta \in D_{bulk}$ and $1 \leq i \leq N$, we have
\begin{equation}\label{finite_rank_1}
|\ud{\tb G}-\ud{\tbhat \gbra}|, \quad |\ud{\tbhat \gbra \tbhat}-\ud{\tb G \tb}| \prec \frac{1}{N \eta}.
\end{equation}
Furthermore, we have the following upper bounds
\begin{equation}\label{trivial_4}
|\ud{\tbhat \gbra}|, \qquad |\ud{\tbhat \gbra \tbhat}| \prec 1.
\end{equation}
Moreover, we have
\begin{equation}\label{concentrate_4}
|\IE  \ud{\tbhat \gbra}| ,\quad |\IE \ud{\tbhat \gbra \tbhat}| \prec \frac{1}{N \eta}.
\end{equation}
\end{lemma}

The strategy in \cite{eta1} to prove the local laws (\ref{local_law_lemma}) is to use Gaussian integration by parts in combination with Lemma \ref{lemma_concentrate} to find a pair of equations for $S_i^{[1]}$ and $T_i^{[1]}$, and thus obtain the following estimates:
\begin{equation}\label{bound_S_1}
S_i^{[1]}=-\frac{z-\omega_B(z)}{a_i-\omega_B(z)}+O_{\prec}(\Psi);  \qquad T_i^{[1]}=O_{\prec}(\Psi).
\end{equation}
In this section, we extend this technique to deal with the quantity in (\ref{K_B_formula}),
$$K_{B,2}(z_1,z_2):=\frac{1}{N} \sum_{j=1}^N \frac{1}{a_j-\omega_B(z_1)} (G(z_2) \tb G(z_1))_{jj} .$$
It suffices to find a local law for the two point function $(G(z_2) \tb G(z_1))_{jj}$, as well as $(\tb G(z_2) \tb G(z_1))_{jj}$. For simplicity, recall the shorthands (\ref{le shorthands}),
$$F \equiv G(z_2)\,,\qquad G \equiv G(z_1)\,.$$
In addition, for $z_1=E_1+\ii\eta_1$, $z_2=E_2+\ii\eta_2 \in \C \setminus \R$, we define two control parameters
\begin{equation}\label{control_para}
\Xi_1\equiv \Xi_1(z_1,z_2):= \frac{1}{\sqrt{N |\eta_1|}| \eta_2|} +\frac{1}{\sqrt{N |\eta_2|} |\eta_1|}; \qquad \Xi_2 \equiv \Xi_2(z_1,z_2):=\frac{1}{N \eta_1^2}+\frac{1}{N |\eta_1\eta_2|}.
\end{equation}
The following lemma, whose proof is given in the Appendix, ensures that one can replace $F, G$ by $\fbra$ and $\gbra$ respectively with affordable price.
\begin{lemma}\label{remove}
For all $1 \leq i,j \leq N$ and $z_1, z_2 \in D_{bulk}$, we have
\begin{equation}\label{removei}
\Big| (F \tb G)_{jj}-( F^{(i)}\tb^{(i)} G^{(i)})_{jj}\Big| , \quad \Big| (\tb F \tb G)_{jj}-(\tbbra F^{(i)}\tb^{(i)} G^{(i)})_{jj}\Big| \prec \Xi_1.
\end{equation}
The same holds true for all $z_1, \overline{z_2} \in D_{bulk}$.
\end{lemma}
With Lemma \ref{remove}, we can reduce the problem to study the approximation
$$\widehat{K}_{B,2}(z_1,z_2):= \frac{1}{N} \sum_{i=1}^N \frac{1}{a_i-\omega_B(z_1)} (\fbra \tbbra \gbra)_{ii}.$$
Coming in pair with $(\fbra \tbbra \gbra)_{ii}$, we will also study $(\tbbra \fbra \tbbra \gbra)_{ii}$. Note that
\begin{align}\label{approximate_qq}
\Big( \tb^{(i)} F^{(i)}\tb^{(i)} G^{(i)}\Big)_{ii}=&\ei^* (1-\ei^*\ei-\gi^*\ei-\ei^*\gi-\gi^*\gi) \tb^{\<i\>} (1-\wi^*\wi) F^{(i)} \tbbra G^{(i)}\ei \nonumber\\
=&-\gi^* \tb^{\<i\>}F^{(i)}\tb^{(i)} G^{(i)} \ei +O_{\prec}\Big( \frac{1}{\sqrt{N |\eta_1 \eta_2|}} \Big).
\end{align}
The last step follows from Lemma \ref{previous_bound} and from the Cauchy-Schwarz inequality, i.e.,
\begin{equation}\label{trivial_1}
|\wi^* \fbra \tbbra \gbra \ei | \leq \|\tbbra\|_{\mathrm{op}} \|\wi^* \fbra \|_2 \|\gbra \ei\|_2 =O_{\prec}\Big( \frac{1}{\sqrt{|\eta_1 \eta_2|}}\Big).
\end{equation}
Combining with Lemma \ref{remove}, we have 
\begin{equation}\label{approximate}
(\tb F \tb G)_{ii}=(\tbbra \fbra \tbbra \gbra)_{ii}+O_{\prec}(\Xi_1):=-S_i^{[2]}+O_{\prec}(\Xi_1),
\end{equation}
where  we define the following two point functions for simplicity
\begin{equation}\label{define_S_two}
S_i^{[2]}:=\gi^* \tb^{\<i\>}F^{(i)}\tb^{(i)} G^{(i)} \ei, \qquad T_i^{[2]}:=\gi^* F^{(i)}\tbbra G^{(i)} \ei.
\end{equation}
It is hence enough to look at $S^{[2]}_i$. Using Lemma \ref{previous_bound}, it is straightforward to check the crude bound
\begin{equation}\label{trivial_2}
|S_i^{[2]}|, \quad |T_i^{[2]}| \prec \frac{1}{\sqrt{|\eta_1\eta_2|}}.
\end{equation}
As an analogue of Lemma \ref{lemma_concentrate}, we have the following concentration results for the two point functions $S^{[2]}_i$, $T_i^{[2]}$ and $(\fbra \tbbra \gbra )_{ii}$. 
\begin{lemma}\label{concentrate_2}
The following hold uniformly for all $z_1,z_2 \in D_{bulk}$ or $z_1,\overline{z_2} \in D_{bulk}$:
$$|\IE S_i^{[2]}| \prec \Xi_1; \qquad |\IE T_i^{[2]}| \prec \Xi_1; \qquad |\IE (\fbra \tbbra \gbra )_{ii}| \prec \Xi_1.$$
\end{lemma}
The proof can be found in the Appendix. Together with these concentration results, we use Gaussian integration by parts to find a pair of linear equations of $\E_\gi S_i^{[2]}$ and $\E_\gi T_i^{[2]}$, and then solve for $\E_\gi S_i^{[2]}$. Note that if $g \sim {N}_\C(0,\sigma^2)$, since $g$ and $\bar{g}$ are independent, then we have the formula of integration by parts, 
\begin{equation}\label{integration_by_parts_complex}
\int_{\C} \bar{g} f(g,\bar{g}) e^{-\frac{|g|^2}{\sigma^2}} \dd g  \wedge \dd \bar{g}=\sigma^2 \int_{\C} \partial_{g} f(g,\bar{g}) e^{-\frac{|g|^2}{\sigma^2}} \dd g  \wedge \dd \bar{g}.
\end{equation}
By direct computation and 
\begin{equation}\label{derivative_R_complex}
\frac{\partial W_i}{\partial g_{ik} }= -\ek (\ei+\gi)^*,
\end{equation}
we obtain that
\begin{align}\label{counterpart}
\E _{\gi} S^{[2]}_i=&\E_{\gi}\Big[ \underline{\tbhat \fbra} \Big( S_i^{[2]} -b_i T_i^{[2]} +(b_i \gi^* \ei+\gi^* \tbhat \gi) ( (\fbra \tbbra \gbra)_{ii}+T^{[2]}_i) \Big)\nonumber \\
&+\underline{\tbhat \fbra W_i \tbhat} \Big( (\fbra \tbbra \gbra)_{ii} +T_i^{[2]} \Big)\nonumber\\
&+(\underline{\tbhat \fbra \tbbra \gbra} - \underline{\tbhat \fbra}) \Big( S_i^{[1]} -b_i T_i^{[1]} +(b_i \gi^* \ei+\gi^* \tbhat \gi) ((\gbra)_{ii} +T_i^{[1]}) \Big) \nonumber\\
& +(\underline{\tbhat \fbra \tbbra W_i \gbra \tbhat}-\underline{\tbhat \fbra W_i \tbhat}) \Big( (\gbra)_{ii} +T_i^{[1]} \Big)\Big].
\end{align}
Note that $\|\fbra\|_{\mathrm{op}} \leq \frac{1}{\eta_2}$, which implies that
$$|\ud{\tbhat \fbra W_i \tbhat} -\ud{\tbhat \fbra \tbhat} |=\frac{1}{N} | \wi^* (\tbhat)^2 \fbra \wi| \leq \frac{1}{N} \|\tbhat\|^2_{\mathrm{op}} \|\fbra\|_{\mathrm{op}} \|\wi\|_2\prec \frac{1}{N |\eta_2|};$$ 
and similarly
$$|\ud{\tbhat \fbra \tbbra \gbra W_i \tbhat} -\ud{\tbhat \fbra \tbbra \gbra \tbhat} |=\frac{1}{N} |\wi^* (\tbhat)^2  \fbra \tbbra \gbra \wi| \prec \frac{1}{N |\eta_1 \eta_2|}.$$
In addition, we have the following lemma, with proof provided in the Appendix.
\begin{lemma}\label{finite_rank_2}
The following hold uniformly for all $z_{1}, z_2 \in D_{bulk}$ or $z_{1}, \overline{z_2} \in D_{bulk}$:
\begin{equation}\label{finite_rank_3}
|\ud{\tbhat \fbra \tbbra \gbra}-\underline{\tb F \tb G}|, \quad |\ud{\tbhat \fbra \tbbra \gbra \tbhat}-\underline{\tb F \tb G \tb}| \prec \Xi_2.
\end{equation}
Furthermore, we have
\begin{equation}\label{trivial_bounds}
|\ud{\tbhat \fbra \tbbra \gbra}|, \quad |\ud{\tbhat \fbra \tbbra \gbra \tbhat} |\prec \frac{1}{\sqrt{|\eta_1\eta_2|}}.
\end{equation}
In addition, we have
\begin{equation}\label{concentrate_5}
| \IE \underline{\tbhat \fbra \tbbra \gbra}|, \quad | \IE \underline{\tbhat \fbra \tbbra \gbra \tbhat}| \prec \Xi_2.
\end{equation}
\end{lemma}
Combining with Lemma \ref{previous_bound}, (\ref{trivial_1}), (\ref{trivial_2}), (\ref{trivial_4}) and (\ref{trivial_bounds}), we obtain that
$$\E _{\gi} S^{[2]}_i=\E_{\gi}\Big[ \underline{\tbhat \fbra} \Big( S_i^{[2]} -b_i T_i^{[2]} \Big) +\underline{\tbhat \fbra \tbhat} \Big( (\fbra \tbbra \gbra)_{ii} +T_i^{[2]} \Big)$$
$$-\Big( \underline{\tbhat \fbra \tbbra \gbra} - \underline{\tbhat \fbra} \Big) (\tbbra \gbra)_{ii}  $$
$$+\Big(\underline{\tbhat \fbra \tbbra \gbra \tbhat}  -\underline{\tbhat \fbra \tbhat} \Big) (\gbra)_{ii}  \Big]+O_{\prec}(\Xi_1+\Xi_2).$$

Using the concentration results (\ref{concentrate_4}) and (\ref{concentrate_5}), together with approximation results (\ref{removei}), (\ref{removei_1}), (\ref{finite_rank_1}) and (\ref{finite_rank_3}), we have
\begin{align}\label{use_later_S}
\E _{\gi} S^{[2]}_i =&\underline{\tb F}  \E_{\gi} \Big( S_i^{[2]} -b_i T_i^{[2]} \Big) +\underline{\tb F \tb}  \E_{\gi} \Big( (F \tb G)_{ii} +T_i^{[2]}\Big)\nonumber\\
&+(\underline{\tb F \tb G \tb}  -\underline{\tb F \tb}) \E_{\gi} G_{ii}  - (\underline{\tb F \tb G}- \underline{\tb F}) \E_{\gi} (\tb G)_{ii}+O_{\prec}(\Xi_1+\Xi_2).
\end{align}
We use integration by parts on $T_i^{[2]}$ and obtain similarly that
\begin{align}\label{use_later_T}
\E _{\gi} T^{[2]}_i =&\underline{ F} \E_{\gi} \Big( S_i^{[2]} -b_i T_i^{[2]} \Big) +\underline{ F \tb} \E_{\gi} \Big( (F \tb G)_{ii} +T_i^{[2]}\Big) \nonumber\\
&+(\underline{ F \tb G \tb}  -\underline{F \tb} ) \E_{\gi} G_{ii} - (\underline{ F \tb G}- \underline{ F}) \E_{\gi} (\tb G)_{ii}+O_{\prec}(\Xi_1+\Xi_2).
\end{align}

Combining them together, we have
\begin{align}
 \underline{F} \E_{\gi} S^{[2]}_i =&  -\underline{\tb F} \E_{\gi}(F \tb G)_{ii} +\mathcal{Y} \E_{\gi}\Big( (F \tb G)_{ii}+T_i^{[2]}\Big)\nonumber\\
&+\Big( \underline{ F} (\underline{\tb F \tb G \tb}  -\underline{\tb F \tb} )- \underline{\tb F} (\underline{ F \tb G \tb}  -\underline{F \tb}) \Big) \E_{\gi} G_{ii}  \nonumber\\
& -\Big( \ud{F} (\underline{\tb F \tb G}- \underline{\tb F})  - \ud{\tb F}(\underline{ F \tb G}- \underline{ F})\Big)  \E_{\gi} (\tb G)_{ii}+O_{\prec}(\Xi_1+\Xi_2),\label{eq_FS}
\end{align}
where $\mathcal{Y}:=\ud{\tb F}-(\ud{\tb F})^2+\ud{\tb F \tb} \cdot \ud{F}$.
We recall the following result:
\begin{lemma}[Lemma 5.1 in \cite{eta2}]\label{mathcal_Y}
Under the assumptions in Proposition \ref{prop}, the estimate
$$\ud{\tb G(z)}-(\ud{\tb G(z)})^2+\ud{\tb G(z) \tb} =O_{\prec}\Big( \frac{1}{N \eta}\Big),$$
holds uniformly for $z=E+\ii \eta \in D_{bulk}$.
\end{lemma}

\begin{remark} The smallness of $\mathcal{Y}$ in~Lemma~\ref{mathcal_Y} is a consequence of the optimal concentration estimates established in~\cite{eta2} and the fact that $\E[\mathcal{Y}]=0$. The latter can be seen from Ward identities similar to Subsection~\ref{subsection ward}.
\end{remark}

Summing over $i$, using the concentration results Lemma \ref{lemma_concentrate}, \ref{concentrate_2} and (\ref{approximate}), we have
$$\ud{F} \cdot \underline{\tb F \tb G} = \underline{\tb F} \cdot \underline{F \tb G}  -   \underline{ F} \cdot \Big(\underline{\tb F \tb G \tb}  \cdot \underline{ G} -\underline{\tb F \tb} \cdot \underline{ G}   -(\underline{\tb F \tb G}- \underline{\tb F}) \underline{\tb G} \Big) $$
$$+\underline{\tb F}  \Big(\underline{ F \tb G \tb} \cdot \underline{ G} -\underline{F \tb} \cdot \underline{ G} - (\underline{ F \tb G}- \underline{ F}) \underline{\tb G} \Big)+O_{\prec}(\Xi_1+\Xi_2) .$$
We can solve for $\underline{\tb F \tb G \tb}$ and 
$$\underline{\tb F \tb G \tb}=\Big( -\frac{1}{\g} +\frac{\underline{\tb G}}{\underline{G}} + \frac{\underline{\tb F}}{\underline{F}}\Big) \underline{\tb F \tb G} +\frac{\underline{\tb F}}{\underline{F}\cdot \underline{G}} \Big( 1-\underline{\tb G}\Big) \underline{F \tb G} +\underline{\tb F \tb}-\frac{(\underline{\tb F})^2}{\underline{F}}+O_{\prec}(\Xi_1+\Xi_2).$$
Returning to (\ref{eq_FS}), we obtain
$$\E_{\gi} S_i^{[{2}]}=-\frac{\underline{\tb F}}{\underline{F}}\E_{\gi} (F \tb G)_{ii}+\Big( -\frac{ \underline{\tb F \tb G}}{\g} +\frac{\underline{\tb F} \cdot \underline{F \tb G}}{ \underline{F}\cdot \g}\Big) \E_{\gi} G_{ii}+O_{\prec}(\Xi_1+\Xi_2).$$
The resolvent definition (\ref{tilde quantities}) and (\ref{approximate}) imply that
$$S_i^{[2]}=-(\tb F \tb G)_{ii}+O_{\prec}(\Xi_1)=(a_i-z_2) (F \tb G)_{ii} -(\tb G)_{ii} +O_{\prec}(\Xi_1).$$
Therefore, we obtain that
$$(a_i-z_2+\frac{\underline{\tb F}}{\underline{F}}) \E_{\gi}(F \tb G)_{ii} =\Big( -\frac{ \underline{\tb F \tb G}}{\g} +\frac{\underline{\tb F} \cdot \underline{F \tb G}}{ \underline{F}\cdot \g}\Big) \E_{\gi} G_{ii} +\E_{\gi}(\tb G)_{ii}+O_{\prec}(\Xi_1+\Xi_2).$$
Dividing $a_i-z_2+\frac{\underline{\tb F}}{\underline{F}} \approx a_i-\omega_B(z_2)$ which is away from zero, combining with the local law Theorem \ref{local}, we have
\begin{align}
 \E_{\gi} (F \tb G)_{ii} =&\Big( -\frac{ \underline{\tb F \tb G}}{\g} +\frac{\underline{\tb F} \cdot \underline{F \tb G}}{ \underline{F}\cdot \g}\Big) \frac{\E_{\gi}G_{ii}}{a_i-\omega_B(z_2)} +\frac{\E_{\gi}(\tb G)_{ii}}{a_i-\omega_B(z_2)}+O_{\prec}(\Xi_1+\Xi_2)\nonumber\\
=&\frac{z_1-\omega_B(z_1)}{(a_j-\omega_B(z_1))(a_j-\omega_B(z_2))}-\frac{1}{m_{fc}(z_1)(a_j-\omega_B(z_1)) (a_j-\omega_B(z_2))} \Big( \ud{\tb F \tb G} -(z_2-\omega_B(z_2)) \ud{F \tb G} \Big)\label{partial_exp_fbg}\\
&+O_{\prec}(\Xi_1+\Xi_2).\nonumber
\end{align}
Armed with Lemma \ref{concentrate_2}, it suffices to estimate $\ud{\tb F \tb G}$ and $\ud{F \tb G}$. The second one is easy to find. It follows from the resolvent identity (\ref{le rock 4}),  the local law Theorem \ref{local} and the arguments in proving (\ref{argument_1}), i.e.,
\begin{equation}\label{fbg}
\underline{F \tb G}= \frac{1}{z_1-z_2} \Big( \underline{\tb G} -\underline{\tb F} \Big) = \frac{T_B}{z_1-z_2}+O_{\prec}(\Xi_2),
\end{equation}
with $T_B \equiv T_B(z_1,z_2)$ given in (\ref{T_B}).

Next, we will use (\ref{partial_exp_fbg}) to solve for $\ud{\tb F \tb G}$. Averaging over $i$ and using the concentration results Lemma \ref{concentrate_2}, we obtain that
\begin{equation}\label{ieq}
L_B \underline{\tb F \tb G}  = (z_1-\omega_B(z_1))m_{fc}(z_1)L_B +(z_2-\omega_B(z_2)) L_B \ud{F \tb G}-m_{fc}(z_1) \ud{F \tb G}+O_{\prec}(\Xi_1+\Xi_2),
\end{equation}
where $L_B \equiv L_B(z_1,z_2)$ is given in (\ref{LM_def}). Next, we show that $L_B$ is non-vanishing. It is straightforward to check from (\ref{LM}) when $z_1$ and $z_2$ are in different half planes. If $z_1$ and $z_2$ are in the same half plane, without loss of generality, we assume $z_1,z_2 \in \C^+$. Since $|\widetilde m'_{fc}(E_0+ \ii 0)| \geq c>0$, due to Lemma \ref{difference}, there exists a neighborhood of $E_0$, denoted by $D_0$, such that for all $z \in D_0 \cap \C^+$, $|m'_{fc}(z)| \geq c>0$ for large $N$. Hence $m_{fc}(z)$ is locally injective in such neighborhood, 
so are $\omega_A(z)$ and $\omega_B(z)$. Thus for all $z_1,z_2 \in D_0 \cap \C^+$, $L_B(z_1,z_2) \neq 0$, and (\ref{ieq}) implies that
\begin{equation}\label{bfbg}
\underline{\tb F \tb G} = (z_1-\omega_B(z_1))m_{fc}(z_1) +(z_2-\omega_B(z_2)) \frac{T_B}{z_1-z_2}-\frac{m_{fc}(z_1)}{z_1-z_2} \frac{T_B}{L_B}+\frac{1}{|L_B|}O_{\prec}(\Xi_1+\Xi_2).
\end{equation}
Note that if $z_1$ and $z_2$ are in different half planes, $\frac{1}{L_B(z_1,z_2)} =\frac{\omega_B(z_1)-\omega_B(z_2)}{m_{fc}(z_1)-m_{fc}(z_2)} \sim 1$ because of (\ref{imaginary_bound}) and (\ref{m_bound}). If $z_1$ and $z_2$ are in the same half plane, then $\omega_B(z)$ is analytic in a neighborhood encircling $z_1,z_2$. Due to (\ref{imaginary_bound}) and (\ref{m_bound}), we have
$$|L_B(z_1,z_2)-L_B(z_1,z_1)|=\Big|\sum_{j=1}^N \frac{1}{a_j-\omega_B(z_1)}\Big( \frac{1}{a_j-\omega_B(z_1)}-\frac{1}{a_j-\omega_B(z_2)} \Big) \Big|=O_{\prec}(|z_1-z_2|).$$
In addition, by (\ref{LMprime}), (\ref{omega_prime_bound}) and the fact that $|\wt m'_{fc}(E_0+\ii 0)| > c>0$, for all $z_1 \in D_0 \cap \C^+$ we have $|L_B(z_1,z_1)|=\Big| \frac{m'_{fc}(z_1)}{\omega'_B(z_1)} \Big| \geq c'>0$. Hence $|L_B(z_1,z_2)| \sim 1$ for sufficiently large $N$. Plugging (\ref{bfbg}) and (\ref{fbg}) into (\ref{partial_exp_fbg}), we obtain that 
$$ (F \tb G)_{jj} = \frac{1}{(z_1-z_2)  (a_j-\omega_B(z_1))(a_j-\omega_B(z_2))} \frac{T_B}{L_B}+O_{\prec}(\Xi_1+\Xi_2).$$  Hence we complete the proof of Proposition \ref{variance_bfbg}.

\section{Expectation of the linear statistics}\label{sec_expectation}
In this section, we compute the expectation of the linear statistics $\Tr f(H_N)$ and prove that the bias for mesoscopic linear statistics in the regular bulk vanishes. Via the Helffer-Sj\"ostrand Calculus, we have
\begin{equation}\label{expectation_fw2}
\E \Tr f(H_N) -N \int_{\R} f(x) \dd \mu_{fc}(x)=\frac{1}{ \pi}  \int_{\Omega_1}  \pzz \tf(z) \Big( \E \Tr G(z) -N m_{fc}(z) \Big) \dd^2z+O_{\prec}(N^{-\tau}).
\end{equation}
We then reduce the bias on the left side to $\E \Tr G(z)-N m_{fc}(z)$. The translation-invariance of Haar measure yields 
\begin{equation}\label{expectation_invariant}
\E [\underline{\tb G} G_{ii}]=\E [\underline{G} (\tb G)_{ii}].
\end{equation}
Then we have
$$\E [\underline{\tb G}] \E[ G_{ii}] +\E \<\underline{\tb G} \> \<G_{ii}\>=\E [\underline{G}] \E[ (\tb G)_{ii}]+\E \< \underline{G}\> \< (\tb G)_{ii}\>,$$
where the local law Theorem \ref{local} implies that 
$$\E [(\tb G)_{ii}] =\frac{\E [\underline{\tb G}]}{\E [\underline{G}]} \E[ G_{ii}] +O_{\prec}(\Psi^3).$$
Combining with the definition of the resolvent (\ref{tilde quantities}), we have
$$\Big(z-a_j -\frac{\E [\underline{\tb G}]}{\E [\underline{G}]}\Big) \E G_{ii}=-1+O_{\prec}(\Psi^3).$$
Since $z-a_j -\frac{\E [\underline{\tb G}]}{\E [\underline{G}]} \approx \omega_B(z)-a_j$ is alway from zero by (\ref{imaginary_bound}), we write
$$\E G_{ii}=\frac{1}{a_i-z +\frac{\E [\underline{\tb G}]}{\E [\underline{G}]}}+O_{\prec}(\Psi^3).$$
Note that Theorem \ref{local} yields
\begin{equation}\label{expansion_error}
\frac{\E [\underline{\tb G}]}{\E [\underline{G}]}=(z-\omega_B(z))+O_{\prec}(\Psi^2).
\end{equation}
Hence, we obtain the following expansion:
\begin{equation}\label{expansion_term}
\E G_{ii}=\frac{1}{a_i-\omega_B(z)} -\frac{1}{(a_i-\omega_B(z))^2} \Big( \frac{\E [\underline{\tb G}]}{\E [\underline{G}]}- z+\omega_B(z)\Big)+O_{\prec}(\Psi^3).
\end{equation}
Summing over the index $i$, we have
$$\E \Tr G(z)=N m_{fc}(z)-N m'_A(\omega_B(z))  \Big( \frac{\E [\underline{\tb G}]}{\E [\underline{G}]}- z+\omega_B(z)\Big)+O_{\prec}\Big(\frac{1}{\sqrt{ N\eta^3}}\Big).$$
Multiplying $\E[ \underline{G}]$ on both sides,
by direct computation, we have
\begin{align*}
\Big( (z-\omega_B(z))m'_A(\omega_B(z)) -m_{fc}(z)\Big) \Big(\E \Tr G-&N m_{fc}(z) \Big)=m'_A(\omega_B(z)) \E [\Tr \tb G]\\
&-m'_A(\omega_B(z)) (z-\omega_B(z)) N m_{fc}(z)+O_{\prec}\Big(\frac{1}{\sqrt{ N\eta^3}}\Big).
\end{align*}
We treat $\mathcal{G}$ in (\ref{tilde quantities}) similarly and obtain corresponding relation by interchanging $A$ with $B$. Combining them together, using the resolvent definition (\ref{tilde quantities}) and subordination equations (\ref{self_m_fc}), we obtain that
$$m^3_{fc}(z) \Delta(z) \Big(\E \Tr G-N m_{fc}(z) \Big) =O_{\prec}\Big(\frac{1}{\sqrt{ N\eta^3}}\Big),$$
with $\Delta(z)$ in (\ref{delta}). Therefore, by (\ref{delta_bound}) and (\ref{imaginary_bound}), we have
$$\E \Tr G(z) -N m_{fc}(z)=O_{\prec}\Big(\frac{1}{\sqrt{ N\eta^3}}\Big).$$
Plugging this into (\ref{expectation_fw2}), using Stokes' formula and (\ref{assumpf}), we obtain
\begin{align*}\E \Tr f(H_N) -N \int_{\R} f(x) \dd \mu_{fc}(x)&=\frac{1}{2 \pi \i}  \int_{\partial \Omega_1} \tf(z) \Big( \E \Tr G(z) -N m_{fc}(z) \Big) \dd z+O_{\prec}(N^{-\tau})\\ &=O_{\prec}\Big( \frac{N^{2 \tau}}{\sqrt{N \eta_0} }\Big)+O_{\prec}(N^{-\tau})\,.
\end{align*}
Thus, in the bulk, the bias vanishes on mesoscopic scales. This concludes the proof of Proposition~\ref{prop2}.

\section{Orthogonal case}\label{sec_orthogonal}
In this section, we prove that Theorem \ref{meso} holds true for the orthogonal conjugation $H_N=A+O B O^{T}$, where $O \in O(N)$ with Haar measure ($\bm\beta=1$). 
\subsection{Proof of Proposition \ref{prop}}
Let $X^T=-X \in \R^{N \times N}$ be deterministic and define $O_t:=e^{t X} O$, $H_{t}:=A+O_t^T B O_{t}$, and $G_t=(H_t-zI)^{-1}$ for $t \in \R$. The function in (\ref{function_invariant}) has constant expectation in $t$ by the translation-invariance of the Haar orthogonal measure. 

Choose $X=\ei \ej^*-\ej \ei^*$ and average over $i$, since $G$ is symmetric, we then obtain an analogue of (\ref{equation_2}), i.e.,
\begin{align*}
\E \Big[  \ea \Big(\g (\tb G)_{jj}-\frac{1}{N} (G \tb G)_{jj}-\underline{\tb G} G_{jj} + \frac{1}{N} (G^2 \tb)_{jj}\Big) \Big] +2 I_j(z)=0,
\end{align*}
with $I_j$ given in (\ref{le Ij}). Proceeding to the arguments in Section \ref{sec_proof_main_theorem} in combination with the local law Theorem \ref{local}, as an analogue of (\ref{sum11}), we obtain
\begin{align}
 & \Big( (z_1-\omega_B(z_1))m'_A(\omega_B(z_1)) -m_{fc}(z_1)\Big) \E [\ea \<\Tr G\>]=m'_A(\omega_B(z_1)) \E [\ea \<\Tr{\tb G}\>]
 -2 \sum_{j=1}^N\frac{I_j(z_1)}{a_j-\omega_B(z_1)}  \nonumber\\
& \qquad + \E \Big[  \ea \Big\< \frac{1}{N} \sum_{j=1}^N \frac{1}{a_j -\omega_B(z_1)} (G \tb G)_{jj}\Big\> \Big]-\E \Big[  \ea \Big\< \frac{1}{N} \sum_{j=1}^N \frac{1}{a_j -\omega_B(z_1)} (G^2 \tb)_{jj}\Big\> \Big] +O_{\prec}\Big(\frac{1}{N \eta_1^2} \Big).\label{sum_orth_1}
\end{align}

Next, we will show that the last line of (\ref{sum_orth_1}) is negligible. For the second term on the last line, using the local law~(\ref{average}) and Cauchy integral formula, we have
$$\E \Big[  \ea \Big\< \frac{1}{N} \sum_{j=1}^N \frac{1}{a_j -\omega_B(z_1)} (G^2 \tb)_{jj}\Big\> \Big] =\E \Big[  \ea \Big\< \frac{1}{N} \sum_{j=1}^N \frac{1}{a_j -\omega_B(z_1)} \frac{\dd}{\dd z_1} (G \tb)_{jj}\Big\> \Big]  =O_\prec \Big(\frac{1}{N \eta_1^2}\Big). $$
For the first term in the last line of (\ref{sum_orth_1}), we need the following  analogue of Proposition \ref{variance_bfbg} for the orthogonal case
\begin{proposition}\label{gbg}
Proposition \ref{variance_bfbg} holds true for $\bm\beta=1$.  In particular, there exists a small neighborhood of $E_0$, denoted by $D_0$, such that for all $z=E+\ii \eta \in D_0 \cap D_{bulk}$ and $1 \leq j \leq N$, we have
\begin{equation}\label{gbg_eq}
 (G(z) \tb G(z))_{jj} = \frac{ [(z-\omega_B(z))m_{fc}(z)]' }{(a_j-\omega_B(z))^2} \frac{1}{L_B(z)}+O_{\prec}\Big(\frac{1}{\sqrt{N \eta} \eta}+\frac{1}{N \eta^2}\Big),
\end{equation}
for sufficiently large $N$, with $L_B(z)$ given in (\ref{LM_def_zz}). 
\end{proposition}
The proof of Proposition \ref{gbg} is provided in the next subsection. Therefore,
$$\E \Big[  \ea \Big\< \frac{1}{N} \sum_{j=1}^N \frac{1}{a_j -\omega_B} (G \tb G)_{jj}\Big\> \Big] =O_{\prec}\Big(\frac{1}{\sqrt{N \eta} \eta}+\frac{1}{N \eta^2}\Big).$$
Returning to (\ref{sum_orth_1}), we have
\begin{align}\label{Io}
&\Big( (z_1-\omega_B(z_1))m'_A(\omega_B(z_1)) -m_{fc}(z_1)\Big) \E [\ea \<\Tr G\>]\nonumber\\
&=m'_A(\omega_B(z_1)) \E [\ea \<\Tr{\tb G}\>]-2 \sum_{j=1}^N\frac{I_j(z_1)}{a_j-\omega_B(z_1)} +O_{\prec}\Big(\frac{1}{\sqrt{N \eta_1} \eta_1}+\frac{1}{N \eta_1^2}\Big).
\end{align}
 We treat similarly for $\tg$ in (\ref{tilde quantities}) and obtain corresponding relation by interchanging $A$ with $B$. 

 Therefore, as an analogue of (\ref{sum1}), we obtain that
$$\E [\ea \<\Tr G\>]=\frac{ \ii 2 \lambda}{\pi} \E \bigg[ \ea  \int_{\Omega_2}  \frac{\partial }{\partial \overline{z_2}} \tf(z_2) \frac{\partial }{\partial {z_2}}K(z_1,z_2) \dd^2 z_2 \bigg]+O_{\prec}\Big(\frac{1}{\sqrt{N \eta_1} \eta_1}+\frac{1}{N \eta_1^2}\Big),$$
with $K(z_1,z_2)$ given in (\ref{kernel_K_1}). Following the arguments in Section \ref{sec_proof_main_theorem}, we complete the proof of Proposition~\ref{prop} for $\bm\beta=1$.

\subsection{Proof of Proposition \ref{gbg}}
In this subsection, we will use partial randomness decomposition to prove Proposition \ref{gbg}. For any Haar orthogonal matrix $O \equiv O_N \in O(N)$,  there exists a random vector $\vi$ uniformly distributed on the unit sphere in $\R^{N}$ and  an independent Haar orthogonal matrix $O^{i}\in O(N-1)$, such that
$$O=-\mbox{sgn}(v_{ii}) R_i O^{\<i\>};\qquad R_i:=I-\ri \ri^*;\qquad \ri:=\sqrt{2} \frac{\ei+\mbox{sgn}(v_{ii}) \vi}{\|\ei+\mbox{sgn}(v_{ii}) \vi\|_2},$$
where $R_i$ is the Householder transform sending $\ei$ to $-\mbox{sgn}(v_{ii}) {\bm v_i}$, and $O^{\<i\>}$ is an orthogonal matrix with $\ei$ as its $i$-th column and $O^{i}$ as its $(i,i)$-minor. Thus we can write
$$\tb=O B O^* =R_i \tb^{\<i\>} R_i, \quad \mbox{with} \qquad \tb^{\<i\>}:=O^{\<i\>} B (O^{\<i\>})^*.$$
Note that $\tbhat$ is independent of $\vi$, and define
$$H^{\<i\>}:=A+\tbhat; \qquad \ghat:=(A+\tbhat-z)^{-1}.$$
Similarly as (\ref{uniform_vector}), there exists $\widetilde \gi \sim {N}_{\R}(0,N^{-1}I_N)$, so that
$$\hi:=\mbox{sgn}(v_{ii})\vi=\mbox{sgn}(v_{ii}) \frac{\widetilde{\gi}}{\|\widetilde \gi\|_2}; \qquad l_i :=\frac{\sqrt{2}}{\|\ei+{\bm h_i}\|_2}=1+O_{\prec}(\frac{1}{\sqrt{N}}); \qquad \ri:=l_i (\ei+{\bm h_i}),$$
We define $g_{ik}:=\mbox{sgn}(v_{ii}) g_{ik} \sim {N}_{\R}(0,N^{-1})$ for $k \neq i$ that are independent and further introduce an independent random variable $g_{ii} \sim {N}_{\R}(0,N^{-1})$. To simplify the proof, we use the Gaussian vector $\gi \sim {N}_{\R}(0, N^{-1}I)$ to approximate $\hi$, and define
$$ \tb^{(i)}=W_i \tb^{\<i\>} W_i; \qquad W_i:=I-{\bm w_i} {\bm w_i}^*; \qquad {\bm w_i}:=\ei+\gi; \qquad G^{(i)}:=(A+B^{(i)}-z)^{-1}.$$
Most results in the unitary case still hold true for orthogonal setup, like Lemma \ref{remove}-\ref{finite_rank_2}. As the analogue of (\ref{integration_by_parts_complex}) and (\ref{derivative_R_complex}), we have
\begin{equation}\label{integration_by_parts_real}
\int_{\R} g f(g) e^{-\frac{g^2}{\sigma^2}} \dd g =\sigma^2 \int_{\R} \partial_{g} f'(g) e^{-\frac{g^2}{\sigma^2}} \dd g;
\end{equation}
\begin{equation}\label{derivative_R_complex}
\frac{\partial W_i}{\partial g_{ik} }= -\ek (\ei+\gi)^*-(\ei+\gi) \ek^*.
\end{equation}
Recall the definitions of $S_i^{[2]}, T_i^{[2]}, S_i^{[1]}, T_i^{[1]}$ in (\ref{define_S_two}) and (\ref{define_S_one}), as well as the control parameters $\Xi_1$ and $\Xi_2$ in (\ref{control_para}). By direct computation, as the counterpart (\ref{counterpart}) in unitary case, we obtain that
\begin{align*}
\E _{\gi} S^{[2]}_i=&\E_{\gi}\Big[ \underline{\tbhat \fbra} \Big( S_i^{[2]} -b_i T_i^{[2]} \Big) +\underline{\tbhat \fbra W_i \tbhat } \Big( (\fbra \tbbra \gbra)_{ii} +T_i^{[2]} \Big)\\
&+(\underline{\tbhat \fbra \tbbra \gbra}- \underline{\tbhat \fbra}) \Big( S_i^{[1]} -b_i T_i^{[1]} \Big) +(\underline{\tbhat \fbra \tbbra \gbra \tbhat}-\underline{\tbhat \fbra \tbhat}) \Big( G^{(i)}_{ii} +T_i^{[1]}  \Big) \Big]\\
&+\frac{1}{N} \E_{\gi} \Big[   (\ei+\gi)^*  \fbra (\tbhat)^2 W_i \fbra  \tbbra \gbra \ei  + (\ei+\gi)^*  \tbhat W_i \fbra \tb^{\<i\>} \fbra \tbbra \gbra \ei\Big]\\
&+\frac{1}{N} \E_{\gi} \Big[   (\ei+\gi)^*  \gbra\tbbra \fbra (\tbhat)^2 W_i \gbra  \ei+ (\ei+\gi)^* \tbhat W_i \gbra  \tbbra \fbra \tbhat \gbra \ei  \Big] \\
&-\frac{1}{N} \E_{\gi} \Big[  (\ei+\gi)^*  \fbra (\tbhat)^2 W_i \gbra \ei+ (\ei+\gi)^* \tbhat W_i \fbra \tb^{\<i\>} \gbra \ei  \Big]+O_{\prec}(\Xi_1+\Xi_2).
\end{align*}

Recalling Lemma A.1 in \cite{eta1}, we have the following rough estimates:
$$\xii^* \fbra (\tbhat)^2 \fbra  \tbbra \gbra \ei,\quad \xii^* \tbhat \fbra \tb^{\<i\>} \fbra \tbbra \gbra \ei,\quad \xii^* \fbra (\tbhat)^2  \gbra \ei  =O_\prec( \Xi_2);$$
$$\xii^* \tbhat \fbra \tb^{\<i\>} \gbra \ei,\quad \xii^* \gbra\tbbra \fbra (\tbhat)^2 \gbra  \ei,\quad \xii^* \tbhat \gbra  \tbbra \fbra \tbhat \gbra \ei  =O_{\prec}( \Xi_2),$$
where $\xii$ stands for either $\ei$ or $\gi$. Therefore, we obtain the same results, e.g., (\ref{use_later_S}) and (\ref{use_later_T}), as in the unitary case with additional affordable error $\Xi_2$. Using the arguments in the proof of Proposition~\ref{variance_bfbg}, one extends Proposition~\ref{variance_bfbg} for $\bm\beta=1$.  In particular, if $z_1=z_2=z \in D_{bulk}$, we have
\begin{align}
&  (G \tb G)_{ii} =\frac{z-\omega_B(z)}{(a_j-\omega_B(z))^2}-\frac{1}{m_{fc}(z)(a_j-\omega_B(z))^2} \Big( \ud{\tb G \tb G} -(z-\omega_B(z)) \ud{G \tb G} \Big)+O_{\prec}(\Xi_1+\Xi_2)\label{partial_exp_fbg_real}.
\end{align}
Averaging over $i$ and by (\ref{LM_def_zz}), we obtain that
$$L_B(z) \underline{\tb G \tb G}  = (z-\omega_B(z))m_{fc}(z) L_B(z) +(z-\omega_B(z)) L_B(z) \ud{G \tb G}-m_{fc}(z) \ud{G \tb G}+O_{\prec}(\Xi_1+\Xi_2).$$
By (\ref{LMprime}), (\ref{omega_prime_bound}) and $\widetilde m'_{fc}(E_0+\ii 0) \neq 0$, there exist some $c>0$ and a neighborhood of $E_0$, denoted by $D_0$, such that for all $z \in D_0 \cap \C^+$, $|L_B(z) |= \Big|\frac{m_{fc}'(z)}{\omega_B'(z)} \Big| >c$ for large $N$. Hence, dividing $L_B(z)$ on both sides, we have
$$ \underline{\tb G \tb G}  = (z-\omega_B)m_{fc}+(z-\omega_B)  \ud{G \tb G}-\frac{m_{fc}(z) }{L_B(z)}\ud{G \tb G} +O_{\prec}(\Xi_1+\Xi_2).$$
Note that the local law Theorem \ref{local} implies that
$$\ud{G \tb G}=\ud{\tb G^2}=\frac{\dd}{\dd z} \ud{\tb G}=[(z-\omega_B)m_{fc}(z)]'+O_{\prec}(\Xi_2).$$
Combining them together, we complete the proof of Proposition \ref{gbg}.

 \subsection{Expectation of the linear statistics and bias}
Theorem \ref{meso} follows directly from Proposition \ref{prop}. The last step is to show that the bias in the regular bulk vanishes. As the analogue of (\ref{expectation_invariant}), we have
\begin{equation}\label{ortho}
\E \underline{G} (\tb G)_{ii}-\frac{1}{N} \E (G \tb G)_{ii}= \E \underline{\tb G} G_{ii} -\frac{1}{N} \E (G^2 \tb)_{ii}.
\end{equation}
Combining with the local law Theorem \ref{local} and the definition of resolvent, we have
$$\Big(z-a_i -\frac{\E [\underline{\tb G}]}{\E [\underline{G}]}\Big) \E G_{ii}=-1-\frac{1}{N} \frac{1}{\E \underline{G}}  \Big( \E (G^2 \tb)_{ii}- \E (G \tb G)_{ii} \Big)+O_{\prec}(\Psi^3).$$
Dividing $z-a_i -\frac{\E [\underline{\tb G}]}{\E [\underline{G}]}$ on both sides, we obtain that
$$\E G_{ii}=\frac{1}{a_i-z +\frac{\E [\underline{\tb G}]}{\E [\underline{G}]}}+\frac{1}{N} \frac{1}{\E \underline{G}}  \Big(\E ( G^2 \tb)_{ii}- \E (G \tb G)_{ii} \Big)\frac{1}{a_j-z +\frac{\E [\underline{\tb G}]}{\E [\underline{G}]}}+O_{\prec}(\Psi^3).$$
Using (\ref{expansion_error}), we have the following expansion:
$$\E G_{ii}=\frac{1}{a_i-\omega_B(z)} -\frac{1}{(a_i-\omega_B(z))^2} \Big( \frac{\E [\underline{\tb G}]}{\E [\underline{G}]}- z+\omega_B(z)\Big)+\frac{1}{N} \frac{1}{\E \underline{G} (a_i-\omega_B(z))}  \Big( \E ( G^2 \tb)_{ii}-\E (G \tb G)_{ii} \Big)+O_{\prec}(\Psi^3).$$
Summing over the index $i$, by direct computation, we have
\begin{align}\label{eq1_ortho}
\Big( (z-\omega_B(z))m'_A(\omega_B(z)) -m_{fc}(z)\Big)& \Big(\E \Tr G-N m_{fc}(z) \Big)=m'_A(\omega_B(z)) \E [\Tr \tb G]\nonumber\\ &-m'_A(\omega_B(z)) (z-\omega_B(z)) N m_{fc}-S_B(z)+O_{\prec}(N \Psi^3)\,,
\end{align}
where
\begin{equation}\label{I_s}
S_B(z):=\frac{1}{N} \sum_{j=1}^N \frac{1}{(a_j-\omega_B)}  \Big(\E ( G^2 \tb)_{ii}- \E (G \tb G)_{ii} \Big)=:S_{B,1}+S_{B,2}.
\end{equation}
We treat $\mathcal{G}$ in (\ref{tilde quantities}) similarly and obtain that
\begin{align}\label{eq2_ortho}
&\Big( (z-\omega_A(z))m'_B(\omega_A(z)) -m_{fc}(z)\Big) \Big(\E \Tr \mathcal{G}-N m_{fc}(z) \Big) \nonumber\\
&\qquad=m'_B(\omega_A(z)) \E [\Tr \ta \mathcal{G}]-m'_B(\omega_A(z)) (z-\omega_A(z)) N m_{fc}(z)
-S_A(z)+O_{\prec}(N \Psi^3),
\end{align}
where
\begin{equation}\label{J_s}
S_A(z):=\frac{1}{N} \sum_{j=1}^N \frac{1}{(b_j-\omega_A)}  \Big(\E ( \mathcal G^2 \ta)_{ii}- \E (\mathcal G \ta \mathcal G)_{ii} \Big)=:S_{A,1}+S_{A,2}.
\end{equation}
Therefore, combining (\ref{eq1_ortho}) with (\ref{eq2_ortho}), we obtain that
$$m^3_{fc}(z)\Delta(z) \Big(\E \Tr G-N m_{fc}(z) \Big) =m'_B(\omega_A(z)) S_B(z)+m'_A(\omega_B(z)) S_A(z)+O_{\prec}(N \Psi^3),$$
with $\Delta(z)$ in (\ref{delta}). Dividing $m^3_{fc}(z) \Delta(z)$ on both sides and by (\ref{delta_bound}) (\ref{imaginary_bound}) and (\ref{self_m_fc}), we obtain
\begin{align}\label{linear_statistics_formula}
\E \Tr G(z) -N m_{fc}(z)=\frac{\omega_B'(z)}{m_{fc}(z)}S_B(z)+\frac{\omega_A'(z)}{m_{fc}(z)}S_A(z)+O_{\prec}(N \Psi^3).
\end{align}
The first term of $S_B$ is easy to estimate using the local law and Cauchy integral formula, i.e.,
\begin{align}
S_{B,1}=\frac{1}{N}\sum_{j=1}^N \frac{1}{a_j-\omega_B(z)} \frac{\dd}{\dd z} (G \tb)_{jj}
 =(1-\omega'_B(z) ) L_B(z)+ \frac{1}{N}\sum_{j=1}^N \frac{(z-\omega_B(z))\omega_B'(z)}{(a_j -\omega_B(z))^3} +O_{\prec}\Big(\frac{1}{N \eta^2}\Big),\label{ggb_average}
 \end{align}
 with $L_B(z)$ given in (\ref{LM_def_zz}).
By Proposition \ref{gbg}, we have
$$S_{B,2}= \frac{ [(z-\omega_B)m_{fc}(z)]' }{L_B(z)} \frac{1}{N} \sum_{j=1}^N \frac{1}{(a_j-\omega_B(z))^3} .$$
Note that by differentiating (\ref{LM_def_zz}) with respect to $z$, we obtain that
$$\frac{1}{N} \sum_{j=1}^N \frac{1}{(a_j-\omega_B(z))^3}=\frac{L'_B(z)}{2 \omega'_B(z)}.$$
Combining with (\ref{ggb_average}), we have
$$S_B(z) =(1-\omega'_B(z)) L_B(z)+\frac{1}{2}(z-\omega_B(z)) L_B'(z)-\frac{[(z-\omega_B(z))m_{fc}(z)]' L_B'(z)}{ 2  \omega_B'(z)L_B(z)}+O_{\prec}\Big( \frac{1}{\sqrt{N \eta^3}} +\frac{1}{N \eta^3} \Big).$$
Similarly, we obtain the corresponding estimate for $S_A$ by interchanging $B$ with $A$. Therefore, by direct computation from (\ref{linear_statistics_formula}) and (\ref{differential_relation}), we have
$$\E \Tr G(z) -N m_{fc}(z)=b(z)+O_{\prec}\Big( \frac{1}{\sqrt{N \eta^3}} +\frac{1}{N \eta^3} \Big),$$
where $b(z)$ is given by
\begin{equation}\label{easy_b}
b(z):=-\frac{1-\omega_B'(z)}{2} \frac{1}{L_B(z)}  L'_B(z)-\frac{1-\omega_A'(z)}{2} \frac{1}{L_A(z)} L'_A(z)+\frac{m_{fc}'(z)}{m_{fc}(z)}-\frac{m_{fc}'^2(z)}{m_{fc}^3(z)}.
\end{equation}
By (\ref{LMprime}) and (\ref{differential_relation}), we can rewrite the $b(z)$ as
\begin{align}
b(z)&=-\frac{(1-\omega_B'(z)) \omega'_B(z)}{2 m_{fc}'(z)} \Big(\frac{m_{fc}'(z)}{\omega_B'(z)}\Big)' -\frac{(1-\omega_A'(z)) \omega'_A(z)}{2 m_{fc}'(z)} \Big(\frac{m_{fc}'(z)}{\omega_A'(z)}\Big)' +\frac{m_{fc}'(z)}{m_{fc}(z)}-\frac{m_{fc}'^2(z)}{m_{fc}^3(z)}\\
&=\frac{1}{2} \Big(\frac{\omega''_A(z)}{\omega'_A(z)} + \frac{\omega''_B(z)}{\omega'_B(z)} -\frac{m''_{fc}(z)}{m'_{fc}(z)}+\frac{2m'_{fc}(z)}{m_{fc}(z)}\Big)=\frac{1}{2} \frac{\dd}{\dd z} \log\Big( \frac{\omega'_A(z) \omega'_B(z)}{F'_{fc}(z)}\Big).
\end{align}
Using Lemma \ref{bound_bulk} and (\ref{easy_b}), if $z \in D_{bulk}$, we have $|L_B'(z)|=O(1)$ and thus $|b(z)|=O(1)$. Plugging it into (\ref{expectation_fw2}), using the Stokes' formula and (\ref{assumpf}), we have
$$\E \Tr f(H_N) -N \int_{\R} f(x) d \mu_{fc}(x)=\frac{1}{ \pi}  \int_{\Omega_1}  \pzz \tf(z) b(z) \dd^2z+O_{\prec}(N^{-\tau})+O_{\prec}\Big( \frac{1}{\sqrt{N \eta_0}} +\frac{1}{N \eta_0^2} \Big)$$
$$=O_{\prec}\Big( \frac{1}{\sqrt{N \eta}} +\frac{1}{N \eta_0} \Big)+O(\eta_0)+O_{\prec}(N^{-\tau}).$$
Therefore, inside the regular bulk, the bias vanishes on mesoscopic scales, hence we prove Proposition \ref{prop2} and Theorem \ref{meso} for $\bm\beta=1$.

\appendix

\section{}
In this Appendix, we provide proofs for some auxiliary results used in this paper. We start by proving Lemma \ref{difference}.
\begin{proof}[Proof of Lemma \ref{difference}]
We will first look at $|\omega'_B(z)-\omega'_\beta(z)|$. By replacing $F'_{A}(z), F'_{B}(z), \omega_A(z), \omega_B(z), m_{fc}(z)$ by $F'_{\alpha}(z), F'_{\beta}(z), \omega_\alpha(z), \omega_\beta(z), \widetilde{m}_{fc}(z)$ respectively, one defines analogously $\tilde{\Delta}(z)$ as in (\ref{delta}), and (\ref{self_m_fc}), (\ref{system}) also holds true. Thus we have
\begin{align}\label{temp_eq}
|\omega_B'(z)-\omega_\beta'(z)|&=\Big|\frac{m'_{B}(\omega_A(z))}{ \Delta(z)  m_{fc}^2(z)}-\frac{m'_{\beta}(\omega_\alpha(z))}{ \widetilde \Delta(z) \widetilde m_{fc}^2(z)}\Big| \nonumber\\ &\leq \Big|\frac{m'_{\beta}(\omega_\alpha(z))}{ \widetilde \Delta(z) \widetilde m_{fc}^2(z)}-\frac{m'_{\beta}(\omega_\alpha(z))}{ \widetilde \Delta(z)  m_{fc}^2(z)}\Big|+\Big|\frac{m'_{\beta}(\omega_\alpha(z))}{ \widetilde \Delta(z) m_{fc}^2(z)}-\frac{m'_{\beta}(\omega_\alpha(z))}{ \Delta(z)  m_{fc}^2(z)}\Big|\nonumber\\
&\qquad+\Big|\frac{m'_{\beta}(\omega_\alpha(z))}{ \Delta(z) m_{fc}^2(z)}-\frac{m'_{\beta}(\omega_A(z))}{ \Delta(z) m_{fc}^2(z)}\Big|+\Big|\frac{m'_{\beta}(\omega_A(z))}{ \Delta(z) m_{fc}^2(z)}-\frac{m'_{B}(\omega_A(z))}{ \Delta(z)  m_{fc}^2(z)}\Big|.
\end{align}
Note that $|\widetilde m_{fc}(z)|, |\widetilde \Delta(z)| \sim 1, |m'_{\beta}(\omega_\alpha(z))| =O(1)$ for all $z\in D_{bulk}$, and the same bounds $|m_{fc}(z)|$, $|\Delta(z)| \sim 1$, $|m'_{B}(\omega_A(z))|, |m'_{\beta}(\omega_A(z))|=O(1)$ hold for large $N$, because of (\ref{difference_m}). 
Thus in combination of (\ref{difference_m}), the first term on the right side can be bounded by
$$\Big|\frac{m'_{\beta}(\omega_\alpha(z))}{ \widetilde \Delta(z) \widetilde m_{fc}^2(z)}-\frac{m'_{\beta}(\omega_\alpha(z))}{ \widetilde \Delta(z)  m_{fc}^2(z)}\Big| \leq C|m_{fc}(z)-\widetilde m_{fc}(z)| \leq C'(d_{\mathrm{L}}(\mu_{B}, \mu_{\beta})+d_{\mathrm{L}}(\mu_{A}, \mu_{\alpha})).$$
The third term on the right side of (\ref{temp_eq}) can be treated similarly together with $ |\Im \omega_A(z)|, |\Im \omega_\alpha(z)|>c>0$ by (\ref{imaginary_bound}). 
In addition, using $ |\Im \omega_A(z)|>c>0$ by (\ref{imaginary_bound}) and the property of L\'evy distance, one checks that
$$\max_{z\in D_{bulk}}\{ |m'_{\beta}(\omega_A(z))-m'_{B}(\omega_A(z))| \}\leq C d_{\mathrm{L}}(\mu_{B}, \mu_{\beta}).$$
We hence obtain the estimate of the last term of (\ref{temp_eq}). Finally, we estimate the second term of (\ref{temp_eq}). Note that 
$$|\Delta(z) -\widetilde{\Delta}(z)|\leq \Big| \frac{m_A'(\omega_B(z))m_B'(\omega_A(z))}{m^4_{fc}(z)}-\frac{m_\alpha'(\omega_\beta(z))m_\beta'(\omega_\alpha(z))}{\widetilde m^4_{fc}(z)} \Big|$$
$$+\Big|\frac{m_A'(\omega_B(z))}{m^2_{fc}(z)}-\frac{m_\alpha'(\omega_\beta(z))}{\widetilde m^2_{fc}(z)}\Big|+\Big|\frac{m_B'(\omega_A(z))}{m^2_{fc}(z)}-\frac{m_\beta'(\omega_\alpha(z))}{\widetilde m^2_{fc}(z)}\Big|.$$
These terms can be treated similarly as above using (\ref{difference_m}). Hence we obtain the estimate of $|\omega_B'(z)-\omega_\beta'(z)|$. The same holds for $|\omega'_A(z)-\omega'_\alpha(z)|$. Furthermore, by (\ref{differential_relation}) we have
$$ |m'_{fc}(z)-\widetilde m'_{fc}(z)|=|\omega'_B(z)m_A'(\omega_B(z))- \omega'_\beta(z) m'_{\alpha}(\omega_\beta(z))| \leq |\omega'_\beta(z) m'_{\alpha}(\omega_\beta(z))-\omega'_\beta(z) m'_{\alpha}(\omega_B(z))|$$
$$+|\omega'_\beta(z) m'_{\alpha}(\omega_B(z))-\omega'_\beta(z) m'_{A}(\omega_B(z)) |+|\omega'_\beta(z) m'_{A}(\omega_B(z))-\omega'_B(z) m'_{A}(\omega_B(z))|.$$
These terms can also be treated similarly using (\ref{difference_m}) in combination with the estimates on $|\omega_B'(z)-\omega_\beta'(z)|$ and $|\omega_A'(z)-\omega_\alpha'(z)|$. Thus we complete the proof of Lemma \ref{difference}.
\end{proof}

Next, we will check Lemma \ref{remove}. In the rest of this section, we assume for simplicity that $z_1,z_2 \in D_{bulk}$. The proof is not sensitive to whether $z_1, z_2$ belong to the same half plane or different half planes. And it also applies when $z_1,\overline{z_2} \in D_{bulk}$.
\begin{proof}[Proof of Lemma \ref{remove}]
Note that the definition of resolvent (\ref{tilde quantities}) yields that
$$ (\tb F \tb G)_{jj}-(\tbbra F^{(i)}\tb^{(i)} G^{(i)})_{jj} =(z_2-a_j) \Big( ( F \tb G)_{jj}- ( F^{(i)}\tb^{(i)} G^{(i)})_{jj}\Big)  +( \tb G)_{jj}-(\tb^{(i)} G^{(i)})_{jj}.$$
Together with (\ref{removei_1}), the second inequality in (\ref{removei}) follows directly from the first one. It is then sufficient to show the first inequality of (\ref{removei}). From (\ref{h_approximate_r}) and (\ref{bound_d}), we have
\begin{align*}
&\ri=\wi+d_1 \ei+d_2 \gi; \qquad R_i-W_i={\bm w_i} {\bm w_i}^*-\ri \ri^*:=\Delta_i; \\
&\tb-\tbbra=R_i \tbhat R_i-W_i \tbhat W_i=\Delta_i \tbhat W_i +W_i \tbhat \Delta_i+\Delta_i \tbhat \Delta_i .
\end{align*}
Using the second resolvent identity, we have
$$G-\gbra=-G \Big( \tb-\tbbra \Big) \gbra =-G \Big( \Delta_i \tbhat W_i +W_i \tbhat \Delta_i+\Delta_i \tbhat \Delta_i\Big) \gbra :=\delta(G),$$
$$F-\fbra=-F \Big( \tb-\tbbra \Big) \fbra =-F \Big( \Delta_i \tbhat W_i +W_i \tbhat \Delta_i+\Delta_i \tbhat \Delta_i\Big) \fbra:=\delta(F).$$
Combining the definition of the resolvent with the resolvent identity (\ref{le rock 4}), we write
\begin{align}
&{F \tb G}- { F^{(i)}\tb^{(i)} G^{(i)}}= z_1  F G- F A G+F -z_1  \fbra \gbra + \fbra A \gbra- \fbra\nonumber\\
&=\frac{z_1}{z_2-z_1}\Big( F-\fbra -G +\gbra \Big) +  (F-\fbra)- (F A G- \fbra A \gbra). \label{20021401}
\end{align}
The second term can be estimated easily by using the first inequality of (\ref{removei_1}). Furthermore, using the arguments in proving (\ref{argument_1}), one shows that the first term can be bounded by $O_{\prec}\Big( \frac{1}{\sqrt{N \eta_1} \eta_1}+\frac{1}{\sqrt{N \eta_2} \eta_1}\Big)$. It is sufficient to study the last term in (\ref{20021401}). 
\begin{align}
&(F A G)_{jj}-( \fbra A \gbra)_{jj}=( \fbra  A \delta(G))_{jj}+( \delta(F) A \gbra)_{jj}+( \delta(F) A \delta(G))_{jj}\nonumber\\
&=-\ej^* \fbra  A G \Big( \Delta_i \tbhat W_i +W_i \tbhat \Delta_i+\Delta_i \tbhat \Delta_i\Big) \gbra  \ej\nonumber\\& - \ej^* F \Big( \Delta_i \tbhat W_i +W_i \tbhat \Delta_i+\Delta_i \tbhat \Delta_i\Big)  \fbra A  \gbra \ej \nonumber\\
&+\ej^*  F \Big( \Delta_i \tbhat W_i +W_i \tbhat \Delta_i+\Delta_i \tbhat \Delta_i\Big) \fbra  A G \Big( \Delta_i \tbhat W_i +W_i \tbhat \Delta_i+\Delta_i \tbhat \Delta_i\Big) \gbra  \ej   \label{20011402}
\end{align}
The first term of the right side of (\ref{20011402}) is a polynomial of the terms of the following form: 
$$\xii^* \tbhat \gbra \ej =O_\prec(1); \qquad \xii^* \tbhat \xii =O_\prec(1); \qquad \xii^* \gbra \ej =O_\prec(1);$$
$$\ej^* \fbra A G \xii \prec \|A\|_{\mathrm{op}}\|G\|_{\mathrm{op}} \|\fbra \ej\|_2 =O_\prec \Big(\frac{1}{\sqrt{\eta_1}\eta_2}\Big) ; \qquad \ej^* \fbra  A G  \tbhat \xii=O_\prec \Big( \frac{1}{\sqrt{\eta_1} \eta_2 }\Big),$$
where $\xii$ is either $\gi$ or $\ei$, and the coefficients are in the form of $d_1^{k_1}d_2^{k_2} \prec \frac{1}{\sqrt{N}}$ with $k_1+k_2 \geq 1$, $k_1, k_2 \in \N$. Combining (\ref{bound_d}) with the above bounds, we obtain an upper bound of the first term. The second term can be treated similarly. As for the last term, it is a polynomial of the terms above and additional ones:
$$\xii^* \tbhat \fbra A G \yi, \quad \xii^* \tbhat \fbra A G \tbhat \yi, \quad \xii^* \fbra A G \tbhat \yi =O_\prec \Big(\frac{1}{\eta_1\eta_2}\Big),$$
with coefficients in the form of $d_1^{k_1}d_2^{k_2} \prec \frac{1}{N}$ with $k_1+k_2 \geq 2$. Combining with (\ref{bound_d}), it is easy to obtain the first bound in (\ref{removei}) and conclude the proof of Lemma \ref{remove}.
\end{proof}

In the following, we will prove the concentration results for $S_i^{[2]}$ and $T_i^{[2]}$ in Lemma \ref{concentrate_2}. We will use the following large deviation bounds of Gaussian vectors, whose proof is standard.
\begin{lemma}\label{gaussian_deviation}
Let $X=(x_{ij}) \in \C^{N \times N}$ be a deterministic matrix and let ${\bm y}=(y_i) \in \C^N$ be a deterministic vector. For a Gaussian random vector ${\bm g}=(g_1, \cdots, g_N) \sim {N}_{\R}(0,\sigma^2 I_N)$ or ${N}_{\C}(0,\sigma^2 I_N)$, we have
$$|{\bm y}^* {\bm g}| \prec \sigma \|{\bm y}\|_2; \qquad |{\bm g}^* X {\bm g}-\sigma^2 \Tr X| \prec \sigma^2\|X\|_{\mathrm{HS}}.$$
\end{lemma}
Before we proceed to prove Lemma \ref{concentrate_2}, we introduce the following rank-one perturbation formula:
\begin{equation}\label{rank_one_formula}
(D+{\bm\alpha} {\bm\gamma}^*)^{-1}=D^{-1} -\frac{D^{-1} {\bm\alpha} {\bm\gamma}^* D^{-1}}{1+{\bm\gamma}^* D^{-1} {\bm\alpha}},
\end{equation}
for any ${\bm\alpha}, {\bm\gamma} \in \C^N$ and invertible $D \in \C^{N \times N}$. As an application of the rank-one perturbation formula (\ref{rank_one_formula}), we obtain the following trace formula.
\begin{lemma}[Lemma 3.2 in \cite{eta1}]\label{trace_rank}
Let $Q, D \in \C^{N \times N}$ and $D$ be Hermitian. Then for any finite rank Hermitian matrix $R \in \C^{N \times N}$ with rank $r$, we have
$$\Big| \frac{1}{N} \Tr \Big( Q(D+R-z)^{-1} \Big)-\frac{1}{N} \Tr \Big( Q(D-z)^{-1} \Big) \Big| \leq \frac{r \|Q\|_{\mathrm{op}}}{N \eta}.$$
\end{lemma}
We will next use Lemma \ref{trace_rank} to prove Lemma \ref{finite_rank_2}.
\begin{proof}[Proof of Lemma \ref{finite_rank_2}]
We start by showing the first line of (\ref{finite_rank_3}). We first estimate the error of removing the upper index $(i)$. Using the definition of resolvent (\ref{tilde quantities}) and the resolvent identity (\ref{le rock 4}), we have
$$\ud{\tbhat (\fbra \tbbra \gbra- F \tb G)}=z_1 \ud{\tbhat (\fbra  \gbra- F  G)}-\ud{\tbhat (\fbra A \gbra- F A G)}+\ud{\tbhat (\fbra  - F )}$$
$$=\frac{z_1}{z_2-z_1} \ud{\tbhat (\fbra-  \gbra- F  +G)}-\ud{\tbhat (\fbra A \gbra- F A G)}+\ud{\tbhat (\fbra  - F )}.$$
It is easy to check that the last term is bounded by $O_{\prec}\Big(\frac{1}{N \eta_2}\Big)$ using Lemma \ref{trace_rank}, since $H^{(i)}$ is a rank-two perturbation of $H$. Furthermore, the first term is bounded by $O_{\prec}\Big(\frac{1}{N \eta_1^2}+\frac{1}{N \eta_1 \eta_2}\Big)$ using the arguments in the proof of (\ref{argument_1}). For the second term, iterating Lemma \ref{trace_rank} twice, we obtain that
$$|\ud{\tbhat (\fbra A \gbra- F A G)}|=O_\prec\Big( \frac{1}{N \eta_1\eta_2}\Big).$$
Thus we have
$$|\ud{\tbhat (\fbra \tbbra \gbra- F \tb G)}| \prec \frac{1}{N \eta_1^2}+\frac{1}{N \eta_1 \eta_2}+\frac{1}{N \eta_2}.$$
Next, we notice that
\begin{align*}\ud{\tbhat F \tb G}-\underline{\tb F \tb G}=\ud{R_i \tb R_i F \tb G}-\underline{\tb F \tb G}&=\frac{1}{N}\Big(-\ri^* F \tb G \tb \ri-\ri^* \tb F \tb G \ri +(\ri^* \tb \ri )(\ri^* F \tb G \ri)\Big)\nonumber\\ &=O_\prec \Big(\frac{1}{N\eta_1\eta_2}\Big)\,.
\end{align*}
The last step follows the fact that $\|F \tb G \tb\|_{\mathrm{op}} \leq \|\tb\|^2_{\mathrm{op}} \|F\|_{\mathrm{op}}\|G\|_{\mathrm{op}} \leq \frac{C}{\eta_1 \eta_2}$. Hence we prove  the first inequality of (\ref{finite_rank_3}). The second one of (\ref{finite_rank_3}) can be treated similarly. In addition, (\ref{trivial_bounds}) is implied by (\ref{finite_rank_3}) and
\begin{equation}\label{bfbg_later}
|\ud{\tb F \tb G}|=O_\prec \Big(\frac{1}{\sqrt{\eta_1\eta_2}}\Big);\qquad |\ud{\tb F \tb G \tb}|=O_\prec \Big(\frac{1}{\sqrt{\eta_1\eta_2}}\Big).
\end{equation}
Finally, we will prove the concentration inequalities (\ref{concentrate_4}). We will only prove the first one for simplicity. Note that $\tbhat$ is independent of $\gi$, and thus
$$ \IE \underline{\tbhat \fbra \tbbra \gbra}=\IE \Big[ \underline{\tbhat \fbra \tbbra \gbra} -\ud {\tbhat \fhat \tbhat \ghat}\Big].$$
Since $H^{\<i\>}$ is a rank-two perturbation of $H^{(i)}$, using previous arguments and Lemma \ref{trace_rank} repeatedly, we obtain the desired result.
\end{proof}

Now, we are ready to prove Lemma \ref{concentrate_2}.
\begin{proof}[Proof of Lemma \ref{concentrate_2}]
We will only prove the last two inequalities by studying $\IE \xii^* \fbra \tbbra \gbra \ei$, where $\xii$ equals either $\gi$ or $\ei$ for simplicity. Note that by (\ref{approximate_qq}) and the definition of resolvent, we have
\begin{align}
S^{[2]}_i=&(a_i-z_2) (\fbra \tbbra \gbra)_{ii}+( a_i-z_1)\gbra_{ii}-1+O_{\prec}\Big(\frac{1}{\sqrt{N \eta_1 \eta_2}}\Big).
\end{align}
Combining with (\ref{concentrate_1}), the last inequality in Lemma \ref{concentrate_2} implies the first one. Note that
\begin{align}\label{approximate_later}
\xii^* \fbra \tbbra \gbra \ei=&\xii^* \fbra (1-\wi \wi^*) \tbhat (1-\wi \wi^*) \gbra \ei\nonumber\\
=&\xii^* \fbra  \tbhat  \gbra \ei-\xii^* \fbra \wi \wi^* \tbhat \gbra \ei \nonumber\\
&-\xii^* \fbra  \tbhat \wi \wi \gbra \ei+\xii^* \fbra \wi \wi^* \tbhat \wi \wi^* \gbra \ei.
\end{align}
Using the property of $\IE$,
\begin{equation}\label{triangle}
\IE[XY]=\IE[\IE X \IE Y]+\IE X \E_{\gi} Y+\E_{\gi} X \IE Y,
\end{equation}
in combination with (\ref{concentrate_1}) and Lemma \ref{previous_bound}, it is enough to estimate $\IE[ \xii^* \fbra  \tbhat  \gbra \ei]$. Recalling that $G^{\<i\>}$ in (\ref{hat_definition}) is independent of $\gi$, it is hence natural to expand $G^{(i)}$ around $G^{\<i\>}$ and then use Lemma \ref{gaussian_deviation} to obtain the concentration results. However, from the construction in (\ref{hat_definition}),
$$G^{\<i\>}_{ii}(z)=\frac{1}{a_i+b_i-z}$$
is not stable for $z \in D_{bulk}$. To overcome this problem, we further introduce
\begin{align}
H^\bi:=A+\wt B^{(i)}-(b_i+\omega_B(z)-z)\ei\ei^*\,,\qquad G^\bi(z):=\frac{1}{H^\bi-z}\,,\qquad z\in\C^+\,,
\end{align}
to enhance the stability in the $\ei$ direction. Note that from the rank-one perturbation formula Lemma \ref{rank_one_formula}, for $1 \leq i,j \leq N$ we have
\begin{align}\label{hello}
 G_{ij}^{\bi}(z)=G^{(i)}_{ij}(z)+\frac{(b_i+\omega_B(z)-z)G^{(i)}_{ii}(z) G^{(i)}_{ij}(z)}{1-(b_i+\omega_B(z)-z)G^{(i)}_{ii}(z)}=\frac{G^{(i)}_{ij}(z)}{1-(b_i+\omega_B(z)-z)G^{(i)}_{ii}(z)}.
\end{align}
For $z\in D_{bulk}$, we have from the local laws of $G^{(i)}_{ii}(z)$,
\begin{align}\label{le pseudo local laws}
\frac{1}{{1-(b_i+\omega_B(z)-z)G^{(i)}_{ii}}}= \frac{a_i-\omega_B(z)}{a_i-b_i-2\omega_B(z)+z} +O_{\prec}(\Psi)=O (1)\,,
\end{align}
because of (\ref{m_bound}), (\ref{imaginary_bound}) and also $|2 \Im \omega_B(z_1)-\Im z_1| >c >0$ for $z_1 \in D_{bulk}$.
Together with Lemma \ref{previous_bound}, we obtain that
\begin{equation}\label{le pseudo local laws}
G^{\{i\}}_{ii}(z)= \frac{1}{a_i-b_i-2\omega_B(z)+z}+O_{\prec}(\Psi(z))=O_\prec(1); \qquad G^{\{i\}}_{ij} =O_{\prec}(\Psi(z)), \qquad j \neq i.
\end{equation}

Next, we will replace $\fbra$, $\gbra$ by the regularized $F^\bi$ and $G^{\bi}$. As a consequence of the rank-one perturbation formula (\ref{rank_one_formula}), we get for general $\yone$ and $\ytwo$, 
\begin{align}\label{le 2}
 \yone^* G^{(i)} \ytwo &=\yone^* G^\bi \ytwo-\frac{(b_i+\omega_B(z_1)-z_1)\yone^* G^{\bi}\ei \ei^*G^{\bi}\ytwo}{1+(b_i+\omega_B(z_1)-z_1)G_{ii}^\bi}\,.
 \end{align}

Set $\yone^*=\xii^* \fbra \tbhat$ and $\ytwo=\ei$, then
 \begin{align}\label{temp_111}
 \xii^* \fbra \tbhat \gbra \ei=& \xii^* \fbra \tbhat G^\bi \ei-\frac{(b_i+\omega_B(z_1)-z_1) \xii^* \fbra \tbhat G^{\bi} \ei G^{\bi}_{ii}}{1+(b_i+\omega_B(z_1)-z_1)G_{ii}^\bi}\nonumber\\
 =&\frac{ \xii^* \fbra \tbhat G^\bi \ei}{1+(b_i+\omega_B(z_1)-z_1)G_{ii}^\bi}:= \xii^* \fbra \tbhat G^\bi \ei   \Lambda_1(z_1),
 \end{align}
where we have the estimate from~\eqref{le pseudo local laws},
\begin{equation}\label{Lambda_1}
 \Lambda_1(z_1):=\frac{1}{1+(b_i+\omega_B(z_1)-z_1)G_{ii}^\bi}=\frac{a_i-b_i-2\omega_B(z_1)+z_1}{a_i-\omega_B(z_1)}+O_\prec(\Psi(z_1)),
 \end{equation}
Note the $ \Lambda_1(z_1)$ is asymptotically deterministic and $| \Lambda_1(z_1)| \sim 1$, because of (\ref{m_bound}) and (\ref{imaginary_bound}) for $z_1 \in D_{bulk}$. Applying the Cauchy-Schwarz inequality, using the estimates for $F^{(i)}$ in Lemma \ref{previous_bound} and the local law of $G^{\{i\}}_{ij}$ in (\ref{le pseudo local laws}), we have
\begin{equation}\label{CS_ward}
|\xii^* \fbra \tbhat G^{\{i\}} \ei| \leq \|\tbhat\|_{\mathrm{op}} \|\xii^* \fbra  \|_2 \| G^{\{i\}} \ei \|_2 \leq C\|\xii^* \fbra  \|_2 \| G^{\{i\}} \ei \|_2\prec  \frac{1}{\sqrt{\eta_1\eta_2}}.
\end{equation}
Combining (\ref{triangle}), (\ref{temp_111}), (\ref{Lambda_1}) and (\ref{CS_ward}), it suffices to study $\IE [\xii^* \fbra \tbhat G^\bi \ei]$ in order to estimate $\IE [\xii^* \fbra \tbhat \gbra \ei]$. We use the rank-one perturbation formula (\ref{rank_one_formula}) again by letting $\yone=\xii$ and $\ytwo=\tbhat G^\bi \ei$, and have 
 $$\xii^* \fbra \tbhat G^\bi \ei=\xii^* F^\bi \tbhat G^\bi \ei-\frac{(b_i+\omega_B(z_2)-z_2)  \xii^* F^{\bi} \ei \ei^* F^{\bi}\tbhat G^\bi \ei}{1+(b_i+\omega_B(z_2)-z_2)F_{ii}^\bi}.$$
From~\eqref{le pseudo local laws}, we have the estimate
$$\Lambda_2(z_2):=-\frac{(b_i+\omega_B(z_2)-z_2) }{1+(b_i+\omega_B(z_2)-z_2)F_{ii}^\bi}=\frac{(b_i+\omega_B(z_2)-z_2)(a_i-b_i-2\omega_B(z_2)+z_2)}{-a_i+\omega_B(z_2)}+O_\prec(\Psi(z_2)).$$
Note the first term on the right side is deterministic and at constant order since $z'$ in the regular bulk. 
Similarly as (\ref{CS_ward}), we have
$$|\ei^* F^{\bi} \tbhat G^{\bi} \ei| \leq \|\tbhat\|_{\mathrm{op}} \|F^\bi \ei \|_2 \| G^\bi \ei \|_2  \prec \frac{1}{\sqrt{\eta_1\eta_2}}.$$
Combining with (\ref{triangle}) and the concentration results in Lemma \ref{lemma_concentrate}, it is hence enough to estimate
 $$\IE \xii^* F^\bi \tbhat G^\bi \ei.$$
Introduce now the block-diagonal matrix
\begin{align}\label{final_G}
 H^\rir:=A+\wt B^{\li}-(b_i+\omega_B(z)-z)\ei\ei^*\,, \qquad  G^\rir=( H^\rir-z)^{-1}.
\end{align}
Since $G^{\rir}$ is independent of $\gi$, we expand $G^{\bi}$ around $G^{\rir}$ and then apply Lemma \ref{gaussian_deviation}.
Note that
\begin{align}\label{resolvent_entry_Grir}
 G_{ii}^\rir(z)=\frac{1}{a_i-\omega_B(z)}=O(1)\,,\qquad G_{ij}^\rir(z)=0\,,\qquad j\not=i\,.
\end{align}
It is straightforward to check that
\begin{align}
 H^\bi-H^\rir=\wi\si^*+\ti\wi^*\,,
\end{align}
with
\begin{align}
 \wi=\ei+\gi\,,\qquad \si=-\wt B^{\li}\wi\,,\qquad \ti=-(\wt B^\li-\wi^*\wt B^\li \wi I)\wi\,.
\end{align}
Iterating the rank-one perturbation formula~\eqref{rank_one_formula} twice, we obtain the following lemma:
\begin{lemma}[(5.17), (5.22) in \cite{eta1}]\label{le lemma twice rank one}
\begin{align}\label{le twice rank one}
 G^\bi=G^\rir+\frac{\Pi}{1+\Xi_i(z)}\,,
\end{align}
with $\Xi_i(z)$ is given by
\begin{align}
 \Xi_i(z)=\si^*G^\rir\wi+\wi^*G^\rir\ti+\si^*G^\rir\wi \wi^*G^\rir\ti-\wi^*G^\rir\wi \si^*G^\rir\ti\,.
\end{align}
and $\Pi$ is the matrix
\begin{align}\label{le Pi}
 \Pi(z)\equiv\Pi&=(\si^*G^\rir\ti)G^\rir \wi\wi^* G^\rir
 +(\wi^* G^\rir \wi)G^\rir \ti\si^*G^\rir\nonumber\\ &\qquad\qquad-(1+(\si G^\rir\wi))G^\rir \ti\wi^* G^\rir-(1+(\wi^*G^\rir\ti))G^\rir \wi \si^*G^\rir\,.
\end{align}
Furthermore, we have
\begin{align}\label{le bound from eta1 paper}
 |\IE [\Xi_i(z)]|\prec\frac{1}{\sqrt{N\eta}}\,,\qquad \big|\frac{1}{1+\E_\gi [\Xi_i(z)]}\big|\prec 1\,.
\end{align}
Moreover, for $Q_1^{\li}, Q_2^{\li}$ either $\tbhat$ or $I$, 
\begin{equation}\label{eta1_bound}
|\IE [\wi^*  Q_1^{\li} G^{\rir} Q_2^{\li} \wi] |\prec \frac{1}{\sqrt{N \eta}}; \qquad |\wi^*  Q_1^{\li} G^{\rir} Q_2^{\li} \wi| \prec 1;\qquad \wi^* \tbhat \wi =b_i+O_{\prec}\Big(\frac{1}{\sqrt{N}}\Big).
\end{equation}
\end{lemma}
From Lemma~\ref{le lemma twice rank one} we have
\begin{align}
 \xii^* F^{\bi}\wt B^{\li}G^{\bi}\ei&=\xii^*(F^\rir+\frac{\Pi(z_2)}{1+\Xi_i(z_2)})\wt B^{\li}(G^\rir+\frac{\Pi(z_1)}{1+\Xi_i(z_1)})\ei\nonumber\\
&=b_i \xii^* F^\rir \ei G_{ii}^\rir+ b_i\frac{\xii^*\Pi(z_2)\ei G_{ii}^\rir}{1+\Xi_i(z_2)}+ \frac{\xii^*F^\rir \tbhat \Pi(z_1)\ei}{1+\Xi_i(z_1)}+\frac{\xii^*\Pi(z_2)\wt B^\li \Pi(z_1)\ei}{(1+\Xi_i(z_2))(1+\Xi_i(z_1))}\label{20011410}
\end{align}
Take $\IE$ on both sides. Since $G^{[i]}(z)$ is independent of $\gi$, we apply Lemma \ref{gaussian_deviation} and (\ref{resolvent_entry_Grir}) to obtain that
$$|\IE[\gi F^{[i]} \ei]|= |\gi F^{[i]} \ei| =|F^{[i]}_{ii}| |\gi^* \ei| \prec \frac{1}{\sqrt{N}}.$$
Thus the concentration of the first term on the right side of (\ref{20011410}) can be bounded by $O_{\prec}\Big(\frac{1}{\sqrt{N}}\Big)$. The rest terms can also be treated similarly, using the property of $\IE$ (\ref{triangle}), the estimate of the denominator (\ref{le bound from eta1 paper}) and the estimates of numerators of each one as following. Observe that for $z=E+\ii \eta$, 
\begin{align}\label{Pi_ii}
\xii^*\Pi(z) \ei&=(\si^*G^\rir\ti) \xii^* G^\rir \wi\wi^* G^\rir \ei
 +(\wi^* G^\rir \wi) \xii^*G^\rir \ti\si^*G^\rir \ei \nonumber\\ 
 &\qquad\qquad-(1+(\si G^\rir\wi)) \xii^* G^\rir \ti\wi^* G^\rir \ei-(1+(\wi^*G^\rir\ti)) \xii^* G^\rir \wi \si^*G^\rir \ei \,.
\end{align}
Combining with (\ref{eta1_bound}), we hence obtain from the property of $\IE$ (\ref{triangle}) that
\begin{equation}\label{estimate_111}
 |\IE [\xii^*\Pi(z)\ei]| \prec \frac{1}{\sqrt{N \eta}}, \qquad |\xii^* \Pi(z) \ei|\prec 1.
\end{equation}
Next, we look at
\begin{align}\label{hahaha}
\xii^*F^\rir \tbhat \Pi(z_1)\ei &=(\si^*G^\rir\ti)(\xii^* F^\rir\wt B^{\li}G^\rir \wi)(\wi^* G^\rir\ei)
 \nonumber\\ &\qquad+(\wi^* G^\rir \wi)(\xii^* F^\rir\wt B^{\li}G^\rir \ti)(\si^*G^\rir\ei)\nonumber\\ &\qquad-(1+(\si G^\rir\wi))(\xii^* F^\rir\wt B^{\li}G^\rir \ti)(\wi^* G^\rir\ei)\nonumber\\ &\qquad-(1+(\wi^*G^\rir\ti))(\xii^* F^\rir\wt B^{\li}G^\rir \wi)( \si^*G^\rir\ei)\,.
\end{align}
The right side can be written in terms of
$$\wi^*Q^\li G^\rir Q^\li \wi; \qquad F^\rir_{ii} G_{ii}^\rir ; \qquad F^\rir_{ii} G_{ii}^\rir \ei Q^\li \gi; \qquad \gi^* F^\rir\wt B^{\li}G^\rir Q^\li\gi. $$
We will only look at the last term for simplicity. Using Lemma \ref{gaussian_deviation} for $\gi$, we obtain that
\begin{align}
|\IE (\gi^*Q^\li F^\rir\wt B^{\li}G^\rir Q^{\li}\gi )|\prec\frac{1}{\sqrt{N}}\|Q^\li F^\rir\wt B^{\li}G^\rir Q^\li\|_{\mathrm{HS}}\,.
\end{align}
Observe that
\begin{align}
 \|Q^\li F^\rir\wt B^{\li}G^\rir Q^\li\|_{\mathrm{HS}}&\le \| Q^\li\|_{\mathrm{op}} \|Q^\li F^\rir\wt B^{\li}\|_{\mathrm{op}}\,\|G^\rir\|_{\mathrm{HS}} \prec \frac{1}{\eta_2}\big(\frac{\mathrm{Tr}{\im G^\rir}}{\eta_1}\big)^{1/2}.
\end{align}
The last step follows from the modified ward identities
\begin{equation}\label{ward_identity_1}
|G^{[i]}(z)|_{jj}^2=\frac{\Im (G^{[i]}(z))_{jj}}{(1-\delta_{ij})\eta+\delta_{ij} \Im \omega_B(z)}, 
\qquad z=E+\ii \eta.
\end{equation} 
Thus, by (\ref{resolvent_entry_Grir}), we have
\begin{align}
|\IE (\gi^*Q^\li F^\rir\wt B^{\li}G^\rir Q^{\li}\gi )| \prec \frac{1}{\sqrt{N\eta_1}\eta_2}\,.
\end{align}
Note that
$$\E_{\gi}[\gi^*Q^\li F^\rir\wt B^{\li}G^\rir Q^\li\gi] =\frac{1}{N}\mathrm{Tr}(Q^\li F^\rir\wt B^{\li}G^\rir Q^\li) $$
Due to the construction (\ref{final_G}) and (\ref{resolvent_entry_Grir}), we have
\begin{align}
\frac{1}{N}\big|\mathrm{Tr} Q^{\li} G^\rir Q^{\li}-\mathrm{Tr} Q^{\li} \ghat Q^{\li}\big|&= \frac{1}{N}\big|\mathrm{Tr} Q^{\li} G^\rir  (b_i+\omega_B(z)-z) \ei \ei^* \ghat Q^{\li} \big| \nonumber\\
&\prec \frac{1}{N} \| Q^{\li} Q^{\li} G^{[i]} \ei \|_2 \| \ghat\ei\|_2 \prec\frac{1}{N\eta_1}.
\end{align}
Furthermore, we have
\begin{align*}
\Big| \frac{1}{N}\mathrm{Tr}(Q^\li F^\rir\wt B^{\li}G^\rir Q^{\li}) -\frac{1}{N}\mathrm{Tr}(Q^\li \fhat \wt B^{\li} \ghat Q^{\li}) \Big|  \prec \frac{1}{N \eta_1 \eta_2}.
\end{align*}
In addition, since $H^{\<i\>}$ is a Hermitian finite-rank perturbation of $H$, by (\ref{trace_rank}), we have
$$\Big| \frac{1}{N}\mathrm{Tr}(Q^\li F^\rir\wt B^{\li}G^\rir Q^{\li}) -\frac{1}{N}\mathrm{Tr}(Q^\li F \wt B^{\li} G Q^{\li}) \Big| \prec \frac{1}{N \eta_1 \eta_2}.$$
Combining with (\ref{finite_rank_3}) and (\ref{bfbg_later}), we hence obtain a priori bound:
\begin{equation}\label{priori_bound}
\Big|\frac{1}{N}\mathrm{Tr}(Q^\li F^\rir\wt B^{\li}G^\rir Q^{\li}) \Big| \prec \frac{1}{\sqrt{\eta_1 \eta_2}}.
\end{equation}
Thus, we have the following estimate:
$$|\gi^*Q^\li F^\rir\wt B^{\li}G^\rir Q^{\li}\gi | \prec \frac{1}{\sqrt{\eta_1 \eta_2}}.$$
Therefore, using the property of $\IE$ (\ref{triangle}) on (\ref{hahaha}), we have
\begin{equation}\label{estimate_222}
| \IE [\xii^*F^\rir \tbhat \Pi(z_1)\ei ] |\prec \frac{1}{\sqrt{N\eta_1}\eta_2}; \qquad |\xii^*F^\rir \tbhat \Pi(z_1)\ei | \prec \frac{1}{\sqrt{\eta_1\eta_2}}.
\end{equation}
Finally, we treat $\xii^* \Pi(z_2) \tbhat \Pi(z_1)\ei$ in the same way. With above bounds, using the property of $\IE$ (\ref{triangle}) on (\ref{20011410}), we have 
 $$|\IE \xii^* F^\bi \tbhat G^\bi \ei | \prec \frac{1}{\sqrt{N\eta_1}\eta_2}+\frac{1}{\sqrt{N\eta_2}\eta_1},$$
and we hence complete the proof of Lemma \ref{concentrate_2}.
\end{proof}

\end{document}